\documentclass[11pt, reqno]{amsart}
\usepackage[utf8]{inputenc}
\usepackage{amsmath}
\usepackage{amsfonts}
\usepackage{amssymb}
\usepackage{amsthm}
\usepackage{dsfont}
\usepackage{float}
\usepackage{hyperref}
\usepackage{mathdots}
\usepackage{stmaryrd}
\usepackage{tikz-cd}
\usepackage{tikz-3dplot}

\DeclareMathOperator{\Hom}{Hom}
\DeclareMathOperator{\tx}{\tilde{x}}
\DeclareMathOperator{\tit}{\tilde{t}}

\DeclareMathOperator{\disc}{disc}
\DeclareMathOperator{\Gal}{Gal}

\DeclareMathOperator{\End}{End}
\DeclareMathOperator{\Frob}{Frob}
\DeclareMathOperator{\Spec}{Spec}

\DeclareMathOperator{\Jac}{Jac}
\DeclareMathOperator{\HH}{H}

\DeclareMathOperator{\im}{im}
\DeclareMathOperator{\pr}{pr}

\DeclareMathOperator{\lspan}{span}

\theoremstyle{definition}
\newtheorem{definition}{Definition}[section]

\newtheorem{exmpl}[definition]{Example}
\newtheorem{rem}[definition]{Remark}

\theoremstyle{theorem}
\newtheorem{thm}[definition]{Theorem}
\newtheorem{cor}[definition]{Corollary}
\newtheorem{lem}[definition]{Lemma}
\newtheorem{prop}[definition]{Proposition}

\newtheorem{innercustomthm}{Theorem}
\newenvironment{customthm}[1]
  {\renewcommand\theinnercustomthm{#1}\innercustomthm}
  {\endinnercustomthm}
\begin{document}

\title[On the non-Transversality \ldots]{On the non-Transversality of the Hyperelliptic Locus and the Supersingular Locus for $g=3$}
\begin{abstract}
This paper gives a criterion for a moduli point to be a point of non-transversal intersection of the hyperelliptic locus and the supersingular locus in the Siegel moduli stack $\mathfrak{A}_3 \times \mathbb{F}_p$. It is shown that for infinitely many primes $p$ there exists such a point.
\end{abstract}
\author[Pieper]{Andreas Pieper}
\address{
  Andreas Pieper,
  Universität Duisburg-Essen
}
\email{andreas.pieper@uni-due.de}

\maketitle
\tableofcontents

\section{Introduction}
Let $k$ be an algebraically closed field of characteristic $p>2$. Denote by $\mathfrak{A}_g$ the Siegel moduli space parametrizing $g$-dimensional principally polarized abelian varieties over $k$ (this notation is going to be used throughout the article). The study of the intersection $\mathfrak{H}_3\cap \mathcal{S}_3$ of the locus $\mathfrak{H}_3\subset \mathfrak{A}_3$ of Jacobians of smooth hyperelliptic curves and the supersingular locus $\mathcal{S}_3\subset \mathfrak{A}_3$ was initiated by Oort's seminal article \cite{OortHyp}. He showed that this intersection is equidimensional of dimension $1$. The interest in this particular situation arises because it is one of the simplest instances of the series of difficult questions around the intersection of Newton polygon strata and loci in $\mathfrak{A}_g$ defined by Jacobians of certain curves, e.g. the Torelli locus. We recommend the recent survey article by Pries \cite{Pries} and the references therein for the readers interested in this circle of ideas.
\par As Pries observes in her survey, for each prime $p$ the Torelli locus and the supersingular locus intersect for infinitely many values of $g$. But this means that the intersection is non-transversal for infinitely many $g$ because the expected dimension of the intersection is
$$\dim(\mathfrak{A}_g)- \dim(\mathcal{M}_g)- \dim(\mathcal{S}_g)=\frac{g(g+1)}{2} -(3g-3)- \left\lfloor \frac{g^2}{4} \right\rfloor$$
which is negative if $g\geqslant 9$.
\par A natural question is: What is the smallest $g$ for which a Newton polygon stratum intersects a locus defined via curves non-transversally? The answer came quite as a surprise to the author: Already the simplest unexplored case, the $g=3$ hyperelliptic locus and the supersingular Newton stratum, exhibits this phenomenon. At the other extreme, for $g=2$ the only reasonable locus would be $\mathfrak{A}_2^\text{dec}=\im (\mathfrak{A}_1\times \mathfrak{A_1} \rightarrow \mathfrak{A}_2)$ (decomposable principally polarized abelian varieties) and Moret-Bailly \cite[2.5.1 Lemme]{Moret-Bailly} proved that it intersects the supersingular locus transversally.
\par Just to clear things up, we will view $\mathfrak{A}_3$ as a stack, in order to avoid complications with the quotient singularities appearing on the coarse moduli space. The reader who does not want to stick to stacks may instead assume that we add a level-$N$ structure, $N>2,\, p\nmid N,$ making the moduli functor representable by a scheme.
\par The main result of this article looks at smooth hyperelliptic supersingular genus $3$ curve $C$ and a smooth\footnote{Corollary \ref{CorIrredComplSm} describes which irreducible components of the formal neighborhood are smooth} irreducible component $\mathcal{W}$ of the formal neighborhood of the moduli point $[\Jac(C)]$ in $\mathcal{S}_3$. Our theorem will give a necessary and sufficient condition for when $\mathcal{W}$ meets $\mathfrak{H}_3$ non-transversally.
\par To $C$, and $\mathcal{W}$ we will associate a certain geometric configuration in $\mathbb{P}^2$, more precisely: We define a conic $Q$, a line $l$ and a point $P$ on $l$. The geometric origin is the following: $Q$ is the image of the canonical map ${C\rightarrow \mathbb{P}^2}$ of our hyperelliptic curve. $P$ and $l$ will be defined (see the discussion following Definition \ref{DefFil}) via the PFTQ theory of Li and Oort \cite{Li-Oort}. The theorem (Theorem \ref{ThmTransGen} in the main body) is as follows:
\begin{customthm}{A}\label{ThmIntroMain}
The following are equivalent:
\begin{enumerate}
\item[i)] The component $\mathcal{W}$ meets $\mathfrak{H}_3$ non-transversally at $[\Jac(C)]$.
\item[ii)] The line $l$ touches $Q$ at the point $P$.
\end{enumerate}
\end{customthm}
\par In the special case where the $a$-number is one, the theorem simplifies as then the irreducible component $\mathcal{W}$ of $\mathcal{S}_3$ passing through $[\Jac(C)]$ will be unique. Furthermore, the point $P$ and the line $l$ have a convenient description in terms of the Cartier-Manin matrix of $C$. Then Theorem A becomes (see Theorem \ref{ThmTransa1}):
\begin{customthm}{B}\label{ThmIntroMaina1}
Let $C$ be a hyperelliptic supersingular curve of genus $3$ with $a(\Jac(C))=1$. Then the following are equivalent:
\begin{enumerate}
\item[i)] The supersingular locus $\mathcal{S}_3$ meets $\mathfrak{H}_3$ non-transversally at $[\Jac(C)]$.
\item[ii)] There exists a point $P\in C$ such that for one (equivalently: for all) hyperelliptic equation
$$y^2=f(x)$$
of $C$ that satisfies $x(P)=0$ the corresponding Cartier-Manin matrix has the shape
$$\begin{pmatrix}
0 & 0 & 0\\
\ast & 0 & 0\\
\ast & \ast & 0
\end{pmatrix}$$
\end{enumerate}
\end{customthm}
Since the Cartier-Manin matrix is defined to be
$$
M=\begin{pmatrix}
c_{ip-j}
\end{pmatrix}_{i,j=1,\ldots, g}
$$
where the $c_i$ are the coefficients obtained by expanding $f(x)^{\frac{p-1}{2}}= \sum c_i x^i$, the criterion of Theorem \ref{ThmIntroMaina1} is very explicit. As a consequence, we get the following na\"\i ve algorithm for constructing examples: Start by writing down the general genus $3$ hyperelliptic curve where you consider the coefficients $a_i$ as indeterminates:
$$y^2=x^7+ \sum_{i=0}^6 a_i x^i\,.$$
Notice that we can move one Weierstrass point to infinity without losing generality. Then, for fixed $p$, the entries of the Cartier-Manin matrix are polynomials in the $a_i$. Writing down the condition that the Cartier-Manin matrix has the shape as in Theorem \ref{ThmIntroMaina1} ii) leads to an algebraic system of equations in the $a_i$ whose solutions give us candidate curves. After testing the supersingularity and $a(\Jac(C))=1$ (note that $a(\Jac(C))=\dim(\ker(M))$) we are left with curves satisfying the condition of Theorem \ref{ThmIntroMaina1}. Using this technique one quickly finds examples such as the curve $y^2=x^7 + x^5 + 7x^3 + x$ in characteristic $p=11$ or the curve $y^2=x^7 + x^6 + 9x^5 + 5x^4 + 6x^3 + x^2 + 16x + 1$ in characteristic $p=19$. This computation was carried out using \texttt{SAGE} \cite{Sage}. The source code is available on the author's webpage \cite{Source}.
\par To get more general examples we will look at CM-reductions. For this purpose we take the curve over $\mathbb{Q}$
$$C: y^2 = x^7 + 7x^5 + 14x^3 + 7x$$
which has CM by $\mathbb{Q}(\zeta_7+\zeta_7^{-1}, i)$ according to the work of van Geemen, Koike, and Weng \cite[Example 2.1]{GeeKoiWe}. The reduction of $C$ at every prime satisfying the congruences $p\equiv \pm 2 \mod 7,\, p \equiv 3 \mod 4$ satisfies the conditions of Theorem \ref{ThmIntroMaina1} (see Proposition \ref{PropTouch} and Theorem \ref{ThmCMRed} in the main body of the text) . As a consequence we get a moduli point of non-transversal intersection for infinitely many primes. One might wonder whether such a moduli point actually exists for all primes $p\gg 0$; this question is left open.
\par The structure of the article is as follows: Section 2 recalls facts from the literature used in the present paper, mostly about Dieudonn\'e theory and the results of Li \& Oort \cite{Li-Oort} regarding the supersingular locus. In Section 3 we use this theory to describe the variation of cristalline cohomology of the universal $3$-dimensional principally polarized supersingular abelian scheme. This leads to a description of the tangent space to the supersingular locus in $\mathfrak{A}_3$. In Section 4 we deduce Theorems \ref{ThmIntroMain} and \ref{ThmIntroMaina1} from this description. The last section is devoted to the discussion of the CM-example.
\par \textbf{Acknowledgments:} The author thanks Irene Bouw, David Lubicz, Laurent Moret-Bailly, Mathieu Romagny, Jeroen Sijsling, and Stefan Wewers for enlightening and encouraging discussions on the subject of this article, and Rachel Pries for a helpful E-mail correspondence. He is also indebted to Ulrich G\"ortz and the anonymous referee(s) for their constructive feedback which lead to improvements of the manuscript.
\section{Preliminaries}
\subsection{Hyperelliptic curves}
Let us denote by $\mathfrak{H}_g$ the moduli stack of hyperelliptic genus $g$ curves over a field of characteristic $\neq 2$. Furthermore we look at the Torelli morphism
$$t: \mathfrak{H}_g \longrightarrow \mathfrak{A}_g,\, [C] \mapsto [(\Jac(C), \Theta)]\,.$$
The map induced by $t$ on tangent spaces has the following properties:
\begin{lem}\label{LemHypDefo}
Let $C$ be a hyperelliptic curve. The map $t$ induces a map on tangent spaces
$$T_{[C]} \mathfrak{H}_g \longrightarrow T_{\Jac(C)} \mathfrak{A}_g\cong S^2 \HH^0(C, \Omega_C)^\vee$$
which
\begin{enumerate}
\item[i)] is injective
\item[ii)] and its image equals the orthogonal complement of the kernel of the multiplication map
$$S^2 \HH^0(C, \Omega_C)\longrightarrow \HH^0(C, \Omega_C^{\otimes 2})\,.$$
\end{enumerate}
\end{lem}
\begin{proof}
See \cite[p. 223]{ACGII}. The authors of that book assume that the base field is $\mathbb{C}$, however, it is not difficult to see that their proof works over any field of characteristic $\neq 2$.
\end{proof}
The preceding lemma and the Torelli theorem together imply that the map $t:\mathfrak{H}_g \rightarrow \mathfrak{A}_g$ is a locally closed immersion. By abuse of notation we will identify $\mathfrak{H}_g$ with its image in $\mathfrak{A}_g$.
\subsection{Dieudonn\'e modules}\label{SecDieu}
For the rest of the paper we fix the following notation: $k$ is an algebraically closed field of characteristic $p>0$. In \cite{BM}\cite{deJong} (relative) Dieudonn\'e modules are defined. The purpose of this subsection is to review these defininitions and the important theorems. Indeed, let $A$ be a regular local ring essentially of finite type over $k$. Berthelot \& Messing and de Jong consider more general rings, but these assumptions suffice for our purposes.
\par Now \cite[Proposition 1.1.7]{BM} implies the existence of a lift of $A$, i.e., a flat $\mathbb{Z}_p$-algebra $\tilde{A}$ which is $p$-adically complete and satisfies $\tilde{A}/p\cong A$. Furthermore, there exists a lift of Frobenius $\sigma: \tilde{A} \longrightarrow \tilde{A}$ by \cite[Corollaire 1.2.7]{BM}. Notice that $\tilde{A}$ is unique up to isomorphism; however, $\sigma$ is not unique. We choose $\sigma$ once and for all.
\par We consider the continuous differential forms
$$\hat{\Omega}_{\tilde{A}}=\varprojlim_n \Omega_{A_n/\mathbb{Z}_p}$$
where $A_n=\tilde{A}/p^n$. By \cite[Proposition 1.3.1]{BM}, $\hat{\Omega}_{\tilde{A}}$ is a free $\tilde{A}$-module of rank $d=\dim(A)$.
\par Now we give the following definition, taken from \cite[Definition 2.3.4]{deJong}:
\begin{definition}
A \emph{Dieudonn\'e module over $\tilde{A}$} is a finite locally free $\tilde{A}$-module $M$ equipped with an integrable, topologically quasi-nilpotent connection
$$\nabla: M\rightarrow M \hat{\otimes}_{\tilde{A}} \hat{\Omega}_{\tilde{A}} $$
and $\nabla$-horizontal, $\tilde{A}$-linear maps
$$F: M\hat{\otimes}_{\tilde{A}, \sigma} \tilde{A} \longrightarrow M,\, V:M\longrightarrow M\hat{\otimes}_{\tilde{A}, \sigma} \tilde{A}$$
satisfying $F\circ V=V\circ F=p$.
\end{definition}
Here topologically quasi-nilpotent means: For any derivation $\delta\in \Hom_{\tilde{A}}(\hat{\Omega}_{\tilde{A}}, \tilde{A})$ and any $m\in M$, there exists an $n\in \mathbb{N}$ such that $\nabla^{\circ n}_\delta(m)\in pM$.
\par We now recall the following theorem:
\begin{thm}
There is an exact contravariant functor
$$\mathbb{D}: \left\lbrace  p\text{-divisible groups over}\phantom{\tilde{A}} \Spec(A)\right\rbrace \longrightarrow \left\lbrace \text{Dieudonn\'e modules over } \tilde{A}\right\rbrace$$ 
\end{thm}
\begin{proof}
Follows from \cite[Proposition 1.3.3]{BBM}.
\end{proof}
As a special case, one recovers Dieudonn\'e theory over $k$. We will denote the functor
$$\left\lbrace  p\text{-divisible groups over}\phantom{\tilde{A}} \Spec(k)\right\rbrace \longrightarrow \left\lbrace \text{Dieudonn\'e modules over } W(k)\right\rbrace$$ 
by $D$. Throughout this paper we will only use contravariant Dieudonn\'e theory.
\subsection{Polarized flag type quotients}
In this section we recall the results from \cite{Li-Oort} we shall need. We specialize to the case $g=3$, although Li \& Oort treat the case of general $g$ and our techniques would generalize. However, for $g\geqslant 4$ explicit computations with the Dieudonn\'e modules in Li \& Oort's theory get more and more cumbersome as $g$ increases. For that reason we restrict ourselves to $g=3$ in the present paper. The author will revisit the case of $g\geqslant 4$ in the future.
\par Choose once and for all $E/k$, a supersingular elliptic curve defined over $\mathbb{F}_p$ such that the relative Frobenius satisfies the equation $F^2+p=0$. The theory of Li \& Oort gives a bijection between the set of irreducible components of $\mathcal{S}_3$ and polarizations $\eta$ on $E^3$ satisfying $\ker(\eta)=E^3[p]$. The bijection is obtained by writing down families of certain principally polarized quotients of $(E^3, \eta)$, so-called polarized flag-type quotients. The definition is as follows:
\begin{definition}\label{DefPFTQ}
Let $\eta$ be as in the discussion preceding the definition. Let $S$ be a $k$-scheme. A \emph{polarized flag-type quotient} (PFTQ for short) starting in $(E^3,\, \eta)$ is a sequence of isogenies of abelian $S$-schemes
$$E^3\times_k S= \mathcal{Y}_2 \stackrel{\rho_2}{\longrightarrow} \mathcal{Y}_1 \stackrel{\rho_1}{\longrightarrow} \mathcal{Y}_0$$ such that:
\begin{enumerate}
\item[i)] $\ker(\rho_i)$ is an $\alpha$-group scheme of rank $i$, i.e., \'etale-locally isomorphic to $\alpha_p^i$.
\item[ii)] For $i=0,1$ the polarization $\eta$ descends along the $\rho_i$ to polarizations on $\mathcal{Y}_i$ denoted $\eta_i$.
\item[iii)] $\ker(\eta_1) \subset \mathcal{Y}_1[F]$.
\end{enumerate}
\par Two PFTQs
$$E^3\times_k S= \mathcal{Y}_2 \stackrel{\rho_2}{\longrightarrow} \mathcal{Y}_1 \stackrel{\rho_1}{\longrightarrow} \mathcal{Y}_0$$
$$E^3\times_k S= \mathcal{Y}'_2 \stackrel{\rho'_2}{\longrightarrow} \mathcal{Y}'_1 \stackrel{\rho'_1}{\longrightarrow} \mathcal{Y}'_0$$
are said to be isomorphic if there exist isomorphisms $\varphi_i: \mathcal{Y}_i \longrightarrow \mathcal{Y}'_i$ with $i=0,1$ such that
$$\begin{tikzcd}
\mathcal{Y}_2 \arrow[d, equal]\arrow[r, "\rho_2"]&\mathcal{Y}_1 \arrow[d, "\varphi_1"]\arrow[r, "\rho_1"]& \mathcal{Y}_0\arrow[d, "\varphi_0"]\\
\mathcal{Y}'_2 \arrow[r, "\rho'_2"]& \mathcal{Y}'_1 \arrow[r, "\rho'_1"]& \mathcal{Y}'_0
\end{tikzcd}$$
commutes.
\end{definition}
Li \& Oort proved the following result:
\begin{thm}
Let $\eta$ be as above. The functor:
$$(k-\text{Sch})^\text{op}\rightarrow \text{Sets}$$
$$S\mapsto \left\lbrace\text{ PFTQs starting in  } (E^3,\, \eta) \right\rbrace/ \text{iso.}$$
is representable by a smooth projective integral $k$-scheme, denoted $\mathcal{P}_{3, \eta}$.
\par More precisely $\mathcal{P}_{3, \eta}$ admits the following description:
\begin{enumerate}
\item[i)] There is a map $\pi: \mathcal{P}_{3, \eta}\longrightarrow \mathbb{P}^2$ (induced from forgetting $\mathcal{Y}_0$).
\item[ii)] The image of $\pi$ is $\im(\pi)=\mathcal{V}(X_0^{p+1}+X_1^{p+1}+X_2^{p+1})$, i.e., the Hermitian curve. We will denote it by $\mathfrak{C}_H$.
\item[iii)] The map $\pi$ is a $\mathbb{P}^1$-bundle. More precisely, there is an isomorphism
$$\mathcal{P}_{3, \eta} \cong \mathbb{P}_{\mathfrak{C}_H}(\mathcal{O}(1)\oplus \mathcal{O}(-1))$$
\end{enumerate}
\end{thm}
\begin{proof}
See \cite[Lemma 3.7]{Li-Oort} for the representability. The explicit description is \cite[Section 9.4]{Li-Oort}.
\end{proof}
By the Yoneda lemma there is a universal family $\mathcal{Y}_2\rightarrow \mathcal{Y}_1\rightarrow \mathcal{Y}_0$ of PFTQs over $\mathcal{P}_{3, \eta}$. The Abelian scheme $\mathcal{Y}_0$ carries a principal polarization (obtained from descending $\eta$). We denote by
$$\varphi_{\text{LO}}: \mathcal{P}_{3, \eta} \rightarrow \mathfrak{A}_3$$
the corresponding map to the Siegel moduli stack. Clearly the image is contained in the supersingular locus. Now Li \& Oort have proven the following theorem:
\begin{thm}\label{ThmLiOortMap}
\begin{enumerate} \item[i)] The map
$$\varphi_{\text{LO}}: \mathcal{P}_{3, \eta} \rightarrow \mathcal{S}_3$$
contracts a curve $T$ and is quasi-finite on the complement of $T$.
\item[ii)] If we take the disjoint union over the (finitely many) equivalence classes of $\eta$ modulo automorphisms of $E^3$, the obtained map
$$\coprod_{\eta} \mathcal{P}_{3, \eta} \rightarrow \mathcal{S}_3$$
is surjective.
\end{enumerate}
\end{thm}
\begin{proof}
The curve $T$ in i) is defined to be the image of the section
$$s: \mathfrak{C}_H \rightarrow \mathcal{P}_{3, \eta} $$
of $\pi$ corresponding to the map
$$\pr_2:\mathcal{O}(1)\oplus \mathcal{O}(-1)\rightarrow \mathcal{O}(-1)\,.$$
See \cite[pp. 58-59]{Li-Oort} for the details. For ii) we refer to \cite[Corollary 4.2]{Li-Oort}.
\end{proof}
\section{Deformation theory}
\subsection{Dieudonn\'e modules}
In this section we describe the variation of the crystalline cohomology in a Li-Oort family. This is a technical tool we need for computing the Kodaira-Spencer map in the next section.
\par Indeed, let $\mathcal{Y}_0\longrightarrow \mathcal{P}_{3, \eta}$ be the principally polarized abelian scheme constructed by Li \& Oort and $\xi \in \mathcal{P}_{3, \eta}$ be a closed point. Assume that $\xi \notin T$. Our goal is to describe the relative crystalline cohomology of $\mathcal{Y}_0$ in an open neighborhood of $\xi$. To this end, consider the stalk $\mathcal{O}_{\mathcal{P}_{3, \eta}, \xi}$ and denote it by $A$. Choose $\tilde{A},\, \sigma$, a lift of $A$ together with a lift of Frobenius as in Section \ref{SecDieu}. 
Our goal is to describe the Dieudonn\'e module of the Abelian scheme $\mathcal{Y}_0 \times_{\mathcal{P}_{3, \eta}} \Spec(A)$. Intuitively this means that we want to understand the variation of the crystalline cohomology in an open neighborhood of the point $\xi$.
\par 
To start with, recall that we are given a supersingular elliptic curve $E$ over $k$ and a polarization $\eta$ on $E^3$ satisfying $\ker(\eta)=E^3[p]$. Denote by $M_2=D(E^3)$ the Dieudonn\'e module of $E^3$. Then $\eta$ induces an alternating pairing
$$\langle \cdot, \cdot \rangle: M_2^t \times M_2^t \longrightarrow W(k)$$
where $M_2^t$ is the dual Dieudonn\'e module.
\par By \cite[Lemma 6.1]{Li-Oort} there exists a $W(k)$-basis $m_0, Fm_0, m_1, Fm_1,m_2, Fm_2$ of $M_2$ such that for all $i=0,1,2$
$$F\cdot m_i=Fm_i,\, F\cdot Fm_i=-p m_i\, ,$$
and such that the pairing $\langle \cdot, \cdot \rangle$ is given by the matrix
$\begin{pmatrix}
0 & p\\
-p & 0
\end{pmatrix}^{\oplus 3}$ with respect to the basis dual to $m_0, \ldots, Fm_2$.
\par Now consider the constant Abelian scheme $E^3_A=E^3\times \Spec(A)$. We have
$$\mathbb{D}(E^3_A)=M_2\hat{\otimes}_{W(k)} \tilde{A} $$
where the map $W(k)\longrightarrow \tilde{A}$ is the unique map lifting the inclusion $k\hookrightarrow A$ (which exists because $k$ is perfect and $\tilde{A}$ is $p$-adically complete). Denote $\mathbb{M}_2=\mathbb{D}(E^3_A)$. It carries a connection
$$\nabla: \mathbb{M}_2 \longrightarrow \mathbb{M}_2\hat{\otimes} \hat{\Omega}_{\tilde{A}}\,,$$
which equals the trivial connection since $E^3_A$ is a constant family.
\par Before we proceed to determine the Dieudonn\'e module of $\mathcal{Y}_0$, we digress and describe a certain affine open subset of $\mathcal{P}_{3, \eta}$. Indeed, recall that we have a $\mathbb{P}^1$-bundle
$$\pi: \mathcal{P}_{3, \eta} \rightarrow \mathfrak{C}_H$$
where $\mathfrak{C}_H=\mathcal{V}(X_0^{p+1}+X_1^{p+1}+X_2^{p+1}) \subset \mathbb{P}^2$. Let us denote by $U_0\subset \mathfrak{C}_H$ the standard affine open where $X_0$ does not vanish. After relabeling the coordinates we can assume without loss that $\xi\in U_0$.
\par We put $x_i=\frac{X_i}{X_0}$. Now the $\mathbb{P}^1$-bundle $\pi$ has a trivialization over $U_0$ given by the standard trivialization of the vector bundle $\mathcal{O}(1)\oplus\mathcal{O}(-1)$ on $U_0$. This determines an isomorphism
$$\pi^{-1}(U_0)\setminus T\cong U_0\times \mathbb{A}^1\,.$$
We denote the latter affine open by $U\subset \mathcal{P}_{3, \eta}$. By assumption, we have $\xi \in U$. Let us denote by $t$ the function in the coordinate ring of $U$ corresponding to the coordinate on the second factor of $U_0\times \mathbb{A}^1$. By localization we get three elements in $A=\mathcal{O}_{\mathcal{P}_{3, \eta}, \xi}$ which we also denote $x_1,x_2,t$ by slight abuse of notation.
\par We get to the main lemma of this section. 
\begin{lem}\label{LemDieu}
The Dieudonn\'e module $\mathbb{D}(\mathcal{Y}_0\times \Spec(A))$ is given by the\\ \mbox{$\tilde{A}$-submodule} of $\mathbb{M}_2$ generated by
$$m_0+\tx_1^p m_1+\tx_2^p m_2+\tit^p Fm_0\,,-\tx_1^p Fm_0+Fm_1\,, -\tx_2^p Fm_0+Fm_2\,, p \mathbb{M}_2$$
where $\tx_1, \tx_2, \tit\in \tilde{A}$ are lifts of $x_1, x_2, t$ respectively.
\end{lem}
\begin{proof}
Let us denote by $\mathbb{M}_0$ the $\tilde{A}$-submodule of $\mathbb{M}_2$ generated by
$$\gamma_1\stackrel{\textrm{def.}}{=}m_0+\tx_1^p m_1+\tx_2^p m_2+\tit^p Fm_0\,,\gamma_2\stackrel{\textrm{def.}}{=}-\tx_1^p Fm_0+Fm_1\,,$$
$$\gamma_3\stackrel{\textrm{def.}}{=}-\tx_2^p Fm_0+Fm_2\,, p \mathbb{M}_2\,.$$
First we notice that $\mathbb{M}_0$ is independent of the choice of the lifts $\tx_1, \tx_2, \tit$ because $p \mathbb{M}_2$ is contained in $\mathbb{M}_0$.
\par At first glance it might seem unusual that the coordinates are all raised to the $p$th power in the definition of the generators. We will see below that this is necessary in order to satisfy the axioms of a Dieudonn\'e module.
\par To prove that $\mathbb{M}_0$ is indeed a Dieudonn\'e submodule, we must show that it is finite locally free and closed under $\nabla, F, V$. The former assertion follows from the short exact sequence
$$0 \longrightarrow \mathbb{M}_0 \longrightarrow \mathbb{M}_2\longrightarrow \mathbb{M}_2/\mathbb{M}_0 \longrightarrow 0$$
and the observation that $\mathbb{M}_2/\mathbb{M}_0$ is a free $\tilde{A}/p$-module with basis $m_1, m_2, Fm_0$.
\par For the closedness under $\nabla$ we compute
$$\nabla(\gamma_1)=p \tx_1^{p-1} m_1\otimes d\tx_1+p \tx_2^{p-1} m_2\otimes d\tx_2+p \tit^{p-1} Fm_0 \otimes d\tit$$
which is in $p \mathbb{M}_2 \hat{\otimes} \hat{\Omega}_{\tilde{A}}$ and similarly for the other generators. Thus we see that the $p$-th powers in the definition of the generators are essential for the closedness under $\nabla$.
\par We verify now that $\mathbb{M}_2$ is closed under $F$ and $V$. Indeed, it suffices to compute modulo $p \mathbb{M}_2$:
$$F(\gamma_1)\equiv Fm_0+\tx_1^{p^2} Fm_1+\tx_2^{p^2} Fm_2 = (1+\tx_1^{p^2+p}+\tx_2^{p^2+p}) Fm_0 + \tx_1^{p^2} \gamma_2+\tx_1^{p^2} \gamma_3$$
modulo  $p \mathbb{M}_2$. Using the relation $1+x_1^{p+1}+x_2^{p+1}=0$, we obtain $F\gamma_1\in \mathbb{M}_0$. Similarly one has
$$V(\gamma_1)\equiv   (1+\tx_1^{1+p}+\tx_2^{1+p}) Fm_0 + \tx_1 \gamma_2+\tx_1 \gamma_3 \mod p \mathbb{M}_2 $$
and thus $V(\gamma_1)\in \mathbb{M}_0$. Also for the well-definedness of $V$ it was essential that everything is raised to the power $p$. Furthermore, $F\gamma_2, V\gamma_2, F\gamma_3,{ V\gamma_3 \in p\mathbb{M}_2}$. We conclude that $\mathbb{M}_0$ is closed under $F, V$.
\par We have shown that $\mathbb{M}_0$ is a Dieudonn\'e submodule of $\mathbb{M}_2$. The fact that it is equal to the image of
$$\mathbb{D}(\mathcal{Y}_0\times \Spec(A)) \rightarrow \mathbb{D}(E^3_A)=\mathbb{M}_2$$
follows from the explicit description of the $g=3$ Li-Oort family in \cite[Section 9.4]{Li-Oort} (see also \cite[Section 3.2]{KaremakerYobukoYu})
\end{proof}
\subsection{Kodaira-Spencer map}
Let $\mathcal{Y}_0\rightarrow \mathcal{P}_{3, \eta}$ be a Li-Oort family and ${\xi\in \mathcal{P}_{3, \eta}}$ be a closed point as in the previous section. The Kodaira-Spencer map
$$\kappa: T_\xi \mathcal{P}_{3, \eta} \longrightarrow  \Hom^\text{sym}\left(\HH^0(\mathcal{Y}_{0, \xi}, \Omega_{\mathcal{Y}_{0, \xi}}) \longrightarrow \HH^1(\mathcal{Y}_{0, \xi}, \mathcal{O}_{\mathcal{Y}_{0, \xi}})  \right)\,, $$
first introduced in the algebraic setting by Illusie \cite[2.1.5.7]{Illusie}, may be viewed as the differential of the morphism $\phi_{\text{LO}}:\mathcal{P}_{3, \eta} \rightarrow \mathfrak{A}_3$ to the Siegel moduli stack. In this section we will compute the Kodaira-Spencer map.
\par We begin by choosing a basis of $T_\xi \mathcal{P}_{3, \eta}$. Indeed, the first basis element is $\delta_1=\frac{\partial}{\partial t}$. The one-dimensional vector space generated by $\frac{\partial}{\partial t}$ is canonically defined because it is the direction parallel to the fibers of the $\mathbb{P}^1$-bundle. The second basis element we will choose is not canonically defined as it will depend on our choice of a trivialization of this $\mathbb{P}^1$-bundle. Let us take $\delta_2=-x_2^p \frac{\partial}{\partial x_1}+x_1^p \frac{\partial}{\partial x_2}$. Then $\delta_1, \delta_2$ are derivations on the affine open $U\subset \mathcal{P}_{3, \eta}$ that do not vanish at $\xi$.
\par We are now ready to compute the Kodaira-Spencer map. Indeed, consider the $A$-module with connection $\mathcal{H}^1_{\text{dR}}=(\mathbb{M}_0/p, \nabla )$. Then $\mathcal{H}^1_{\text{dR}}$ is isomorphic to the relative de Rham cohomology of $\mathcal{Y}_0 \times_{\mathcal{P}_{3, \eta}} \Spec(A)$ over $\Spec(A)$ and $\nabla$ is identified with the Gauss-Manin connection (see \cite[Proposition 3.6.4]{Ber}).
\par Therefore, it only remains to identify the submodule of differentials. For that purpose we look at the exact sequence
$$\mathcal{H}^1_{\text{dR}} \stackrel{v}{\longrightarrow} (\mathcal{H}^1_{\text{dR}})^{(p)} \stackrel{f}{\longrightarrow}  \mathcal{H}^1_{\text{dR}}\,,$$
where $(\mathcal{H}^1_{\text{dR}})^{(p)}= (\mathcal{H}^1_{\text{dR}})\otimes_{A, x \mapsto x^p} A$ and $f$ (resp. $v$) is the map induced by $F$ (resp. $V$). We will use the following Lemma:
\begin{lem}
There is a locally free, locally direct summand $\omega \subset \mathcal{H}^1_{\text{dR}}$ such that $\omega^{(p)}\subset (\mathcal{H}^1_{\text{dR}})^{(p)}$ is equal to $\im(v)=\ker(f)$. It is uniquely determined by these conditions.
\par Furthermore $\omega= \HH^0(\mathcal{Y}_0 \times_{\mathcal{P}_{3, \eta}} \Spec(A), \Omega_{\mathcal{Y}_0 \times_{\mathcal{P}_{3, \eta}} \Spec(A)})$.
\end{lem}
\begin{proof}
The first assertion is \cite[Proposition 2.5.2]{deJong}. The second assertion follows from \cite[Proposition 4.3.10]{BBM}.
\end{proof}
In our case, we claim that $\omega$ is generated by the four classes
$$g_1 \stackrel{\textrm{def.}}{=} Fm_0+\tx_1\, Fm_1+\tx_2\, Fm_2-p \tit m_0,$$
$$g_2 \stackrel{\textrm{def.}}{=} p\tx_1\,  m_0-p\, m_1,$$
$$p \tx_2\, m_0-p m_2,$$
$$g_3 \stackrel{\textrm{def.}}{=}p\, Fm_0$$
coming from the four elements $V\gamma_1, V \gamma_2, V \gamma_3, V(-p m_0)$ (in the notation of the proof of Lemma \ref{LemDieu}); indeed, it suffices to show that the classes of $p\, Fm_1,\, p\, Fm_2$ in $\omega$ are contained in the span of the four given elements. This follows from the identitites ${p\, Fm_1= p\, \gamma_2 + \tx_1^p p\, F m_0,}$\\
${p Fm_2= p\, \gamma_3 + \tx_2^p p\, F m_0\,.}$
\par We will now assume that $x_2$ does not vanish at $\xi$ which we can, without loss of generality, after switching $x_1, x_2$ if necessary. Under this assumption and using the equation $1+x_1^{p+1}+x_2^{p+1}=0$ we get that the class of
$$p \tx_2\, m_0-p m_2 \equiv - \frac{\tx_1^p}{\tx_2^p} g_2- \frac{1}{\tx_2^p} p \gamma_1 \equiv - \frac{\tx_1^p}{\tx_2^p} g_2\mod p \mathbb{M}_0$$
gives a multiple of $g_2$ in $\omega$. Therefore, $g_1, g_2, g_3$ are a basis for $\omega$ as an $A$-module.
\par Similarly we get the basis
$$c_1\stackrel{\textrm{def.}}{=} p m_0\,, c_2\stackrel{\textrm{def.}}{=}\gamma_2=\tx_1^p\, Fm_0-Fm_1,\, c_3\stackrel{\textrm{def.}}{=}\gamma_1=m_0+\tx_1^p m_1+\tx_2^p m_2+\tit^p Fm_0$$
for $(\mathcal{H}^1_{\text{dR}})/\omega$. Indeed, the class of $\gamma_3= \tx_2^p\, Fm_0-Fm_2$ is redundant thanks to the relation
$$\tx_2^p\, Fm_0-Fm_2 \equiv - \frac{\tx_1}{\tx_2} c_2- \frac{1}{\tx_2^p} g_1 \equiv - \frac{\tx_1^p}{\tx_2^p} c_2\mod \omega $$
using the observation $g_3 = p Fm_0 \in \omega$ and again $1+x_1^{p+1}+x_2^{p+1}=0$.
\par The Kodaira-Spencer map is now calculated as follows:
\begin{lem}\label{LemKS}
$$\kappa(\delta_1)(g_1)=-p m_0=-c_1,$$
$$ \kappa(\delta_1)(g_2)=\kappa(\delta_1)(g_3)=0,$$
$$\kappa(\delta_2)(g_1)=-\tx_2^p\, Fm_1+\tx_1^p\, Fm_2\in \lspan(c_1, c_2),$$
$$\kappa(\delta_2)(g_2)=\tx_2^p pm_0,$$
$$\kappa(\delta_2)(g_3)=0$$
\end{lem}
\begin{proof}
The Kodaira-Spencer map
$$\kappa: T_\xi \mathcal{P}_{3, \eta} \longrightarrow  \Hom^\text{sym}\left(\HH^0(\mathcal{Y}_{0, \xi}, \Omega_{\mathcal{Y}_{0, \xi}}) \longrightarrow \HH^1(\mathcal{Y}_{0, \xi}, \mathcal{O}_{\mathcal{Y}_{0, \xi}})  \right) $$
is computed as follows: For a derivation $\delta$ we look at the map
$$\omega\rightarrow \mathcal{H}^1_\textrm{dR}/\omega$$
induced from $\nabla_\delta$ and restrict it to the fiber at $\xi$. This gives the desired map
$$\kappa(\delta)\in \Hom^\text{sym}\left(\HH^0(\mathcal{Y}_{0, \xi}, \Omega_{\mathcal{Y}_{0, \xi}}) \longrightarrow \HH^1(\mathcal{Y}_{0, \xi}, \mathcal{O}_{\mathcal{Y}_{0, \xi}})  \right)\,.$$
As an example we shall compute $\kappa(\delta_2)(g_1)$ leaving out the other calculations as they are similar. Indeed,
$$\nabla_{-\tx_2^p \frac{\partial}{\partial \tx_1}+\tx_1^p \frac{\partial}{\partial \tx_2}} \left(Fm_0+\tx_1\, Fm_1+\tx_2\, Fm_2-p \tit m_0\right) = -\tx_2^p F m_1+\tx_1^p F m_2$$
and therefore
$$\kappa(\delta_2)(g_1)=-\tx_2^p F m_1+\tx_1^p F m_2\,.$$
To show that this class lies in $\lspan(c_1, c_2)$ we point out that the relation $1+x_1^{p+1}+x_2^{p+1}=0$, the fact $pF m_1\in \omega$, and the non-vanishing of $x_2$ at $\xi$ together imply
$$-\tx_2^p F m_1+\tx_1^p F m_2\equiv \frac{-1}{\tx_2} c_2+ \frac{\tx_1^p}{\tx_2} \left(g_1+\tit c_1\right) \mod \omega$$
And therefore as an element of $(\mathcal{H}^1_{\text{dR}})/\omega$ we have
$$-\tx_2^p F m_1+\tx_1^p F m_2 \in \lspan(c_1, c_2)\,.$$
\end{proof}
\begin{cor}\label{CorUnram}
The map
$$\varphi_{\text{LO}}: \mathcal{P}_{3, \eta} \rightarrow \mathcal{S}_3$$
is unramified on the open set $\mathcal{P}_{3, \eta}\setminus T$.
\end{cor}
\begin{proof}
The differential of $\varphi_{\text{LO}}$ at the point $\xi$ is given by the Kodaira-Spencer map which we computed in the previous lemma. The point $\xi$ was assumed to be in $U$ and satisfied $x_2\neq 0$. The two Kodaira-Spencer classes are linearly independent.
\par This implies that $\varphi_{\text{LO}}$ is unramified on the open set $U\cap \{x_2\neq 0\}$ and similarly for the open sets we obtain by permuting coordinates. Since these open sets cover $\mathcal{P}_{3, \eta}\setminus T$, the lemma follows.
\end{proof}
\subsection{Irreducible components of formal neighborhoods}
Let $Y$ be a supersingular principally polarized abelian variety over $k$. In this section we study the formal completion of $\mathcal{S}_3$ at the moduli point $[Y]$. Let us denote this formal completion by $\hat{\mathcal{S}}_{3,[Y]}$. We show that every irreducible component of $\hat{\mathcal{S}}_{3,[Y]}$ is smooth with one exception when $a(Y)=3$. All the results in this section are well-known to the experts, but we include them here as the author could not find proofs in the literature.
\begin{lem}\label{LemIrredComp}
Let $Y$ be as above. There is a bijection
$$
\left\lbrace
\begin{tabular}{@{}c@{}}
Irreducible components\\
of $\hat{\mathcal{S}}_{3,[Y]}$
\end{tabular} \right\rbrace \longleftrightarrow
\left\lbrace 
\begin{tabular}{@{}c@{}}
Isogenies $E^3 \longrightarrow Y$ which\\
are the composition of a PFTQ
\end{tabular}\right\rbrace_{/\sim}$$
where the equivalence relation $\sim$ is defined as follows: $\pi_1, \pi_2: E^3 \longrightarrow Y$ are called equivalent if there exists an automorphism $\varphi$ of $E^3$ such that
$$\begin{tikzcd}
E^3 \arrow{d}{\varphi}\arrow{dr}{\pi_1}\\
E^3 \arrow{r}{\pi_2} & Y
\end{tikzcd}$$
commutes.
\end{lem}
\begin{proof}
We define a map 
$$f:\left\lbrace 
\begin{tabular}{@{}c@{}}
Isogenies $E^3 \longrightarrow Y$ which\\
are the composition of a PFTQ
\end{tabular}\right\rbrace_{/\sim} \longrightarrow \left\lbrace
\begin{tabular}{@{}c@{}}
Irreducible components\\
of $\hat{\mathcal{S}}_{3,[Y]}$
\end{tabular} \right\rbrace 
$$
as follows: Let $\pi: E^3\longrightarrow Y$ be an isogeny which is the composition of a PFTQ, say
$$E^3\longrightarrow Y_1 \longrightarrow Y_0=Y\,.$$
Denote by $\eta$ the pullback of the principal polarization on $Y$ along $\pi$. As a consequence of \cite[Equation 9.4.11]{Li-Oort} and the properties of the notion of rigid PFTQs in loc. cit. this PFTQ is uniquely determined by $\pi$ unless $\ker(\pi)=E^3[F]$. In the former case one gets a unique point $\xi\in \mathcal{P}_{3, \eta}$ such that $\varphi_{\text{LO}}(\xi)=[Y]$ under the map
$$\varphi_{\text{LO}} : \mathcal{P}_{3, \eta}\longrightarrow \mathcal{S}_3\,.$$
We define $f(E^3 \longrightarrow Y)$ to be the irreducible component $\hat{\mathcal{S}}_{3,[Y]}$ given by the image of the map induced by $\varphi_{\text{LO}}$ on the formal completion at $\xi$.
\par The latter case $\pi=F$ can only happen if $a(Y)=3$. In this case the map
$$\varphi_{\text{LO}} : \mathcal{P}_{3, \eta}\longrightarrow \mathcal{S}_3$$
contracts the curve $T$ to the point $[Y]$. We define $\Xi$ to be the irreducible component of $\hat{\mathcal{S}}_{3,[Y]}$ given by the image of the map induced by $\varphi_{\text{LO}}$ on the formal completion along $T$ and we put $f(E^3 \rightarrow Y)=\Xi$. The map $f$ is well-defined.
\par We have to show that $f$ is bijective. For that purpose we equip $\mathcal{S}_{3}$ with a level $N$ structure for some $N\geqslant 3,\, p\nmid N$ and denote the corresponding moduli space by $\mathcal{S}_{3, N}$. Then, $\mathcal{S}_{3, N}$ is a closed subscheme of the Siegel moduli scheme $\mathcal{A}_{g, N}$ (equipped with a level $N$ structure). Furthermore, as Li \& Oort show \cite[Section 13.14]{Li-Oort} their construction can be equipped with a level structure yielding a map
$$\varphi: \coprod \mathcal{P}_{3, \eta} \rightarrow \mathcal{S}_{3, N}$$
with the same properties as in Theorem \ref{ThmLiOortMap}. The disjoint union is taken over the finite set $\Lambda_N$ of all pairs $(l_N, \eta)$ where $l_N$ is a level $N$ structure on $E^3$ and $\eta$ is a polarization on $E^3$ satisfying $\ker(\eta)=E^3[p]$ (taken modulo automorphisms of $E^3$).
\par Furthermore $\varphi$ is birational by \cite[Remark 6.4]{Li-Oort}. Then, since $\mathcal{S}_{3, N} \longrightarrow \mathcal{S}_3$ is \'etale and $k$ is algebraically closed, it induces an isomorphism on completions. We choose $l_N$, a level $N$ structure on $Y$. It is now enough to describe the irreducible components of the completion of $\mathcal{S}_{3, N}$ at $(Y, l_N)$. In order to treat the cases $a(Y)=3$ and $a(Y) \leqslant 2$ simultaneously we will first define a subset $\Lambda'_N\subseteq \Lambda_N$ and a closed subscheme $X\subseteq \mathcal{S}_{3,N}$ depending on the $a$-number of $Y$. Assume first that $a(Y)=3$; then there is an isomorphism $Y\cong E^3$ as abelian varieties and there is a PFTQ whose composition is  $\pi:E^3\stackrel{F}{\longrightarrow} E^3 \cong Y $. This leads to the irreducible component $\Xi$ defined above. To show that $f$ is a bijection it suffices to show that it restricts to a bijection after removing $\Xi$ from the target set and $f^{-1}(\Xi)$ from the source.
\par To this end we consider $l'_N$ the level $N$ structure on $E^3$ obtained from $l_N$ by pulling back along the isomorphism
$$E^3[N] \longrightarrow Y[N]$$
induced by $\pi$. Furthermore denote by $\eta$ the polarization on $E^3$ defined as the pullback of the principal polarization on $Y$ via $\pi$. Consider the restricted map
$$\coprod_{\Lambda_N\setminus \{(l_N', \eta) \}} \mathcal{P}_{3, \eta} \rightarrow \mathcal{S}_{3, N}$$
and denote its image by $X$. The closed subscheme $X\subset \mathcal{S}_{3,N}$ has one irreducible component less than $\mathcal{S}_{3,N}$. As a last bit of notation we put $\Lambda_N'=\Lambda_N\setminus \{(l_N', \eta) \}$.
\par If instead $a(Y)\leqslant 2$, then we do not have to remove an irreducible component and we put $\Lambda_N'=\Lambda_N$ and $X=\mathcal{S}_{3,N}$.
\par Let us come back to the general situation ($a(Y)$ arbitrary). We have defined in both cases a map, also denoted $\varphi$ by slight abuse of notation
$$\varphi:\coprod_{\Lambda_N'} \mathcal{P}_{3, \eta} \rightarrow X\,.$$
The map $\varphi$ is birational, proper (because $\mathcal{P}_{3, \eta}$ is projective). Now we choose $U\subset X$, a Zariski open neighborhood of $(Y, l_N)$ small enough so that $U$ does not contain any points with $a=3$ except maybe $(Y, l_N)$ itself. Then we look at the cartesian diagram
$$
\begin{tikzcd}
U' \arrow{r}\arrow{d}{\varphi'} & \coprod_{\Lambda_N'} \mathcal{P}_{3, \eta}\arrow{d}{\varphi}\\
U \arrow{r} &X
\end{tikzcd}
$$
The map $\varphi'$ is proper. It is also quasi-finite by Theorem \ref{ThmLiOortMap} because $U$ does not contain any points with $a=3$ except maybe $(Y, l_N)$. Indeed, by construction we removed the only component of $\coprod_{\Lambda_N} \mathcal{P}_{3, \eta}$ containing the unique positive dimensional component of the fiber $\varphi^{-1}(Y, l_N)$. As a consequence we conclude that $\varphi'$ is proper, quasi-finite and thus finite.
\par On the other hand, $U'$ is an open subscheme of the smooth $k$-scheme $\coprod_{\Lambda_N'} \mathcal{P}_{3, \eta}$ and in particular $U'$ is normal. In summary, we have that $\varphi'$ is a finite birational map with normal source. By \cite[Lemma 035Q, (3)]{Stacks} this implies that $\varphi'$ is the normalization of $U$.
\par Now since $U$ is excellent, \cite[Lemma 0C23]{Stacks} implies that normalization commutes with completion. Finally \cite[Lemma 035Q, (2)]{Stacks} gives the desired bijection for the irreducible components of the completion.
\end{proof}
\begin{cor}\label{CorIrredComplSm}
In the notation as above.
\begin{itemize}
\item[i)] If $a(Y)\leqslant 2$, then all the irreducible components of $\hat{\mathcal{S}}_{3,[Y]}$ are smooth.
\item[ii)] If $a(Y)=3$, then all the irreducible components of $\hat{\mathcal{S}}_{3,[Y]}$ are smooth, except one.
\end{itemize}
\end{cor}
\begin{proof}
Let $\mathcal{W}$ be an irreducible component of $\hat{\mathcal{S}}_{3,[Y]}$. Assume that either $a(Y)\leqslant 2$ or $a(Y)=3$ and $\mathcal{W}$ is not the component $\Xi$ discussed in the proof of the previous lemma.
\par Then we know that there exist a polarization $\eta$ on $E^3$ satisfying $\ker(\eta)=E^3[p]$ and a point $\xi\in \mathcal{P}_{3, \eta}$ such that the formal completion $\hat{\mathcal{P}}_{3, \eta, \xi}$ dominates $\mathcal{W}$ via the map
$$\varphi_{LO}: \mathcal{P}_{3, \eta} \longrightarrow \mathcal{S}_3\,.$$
Let us denote $\hat{\mathcal{P}}_{3, \eta, \xi}$ by $\mathcal{W}'$ and the induced map $\mathcal{W}'\rightarrow \mathcal{W}$ by $\varphi$.
\par Then, by our assumptions on $\mathcal{W}$, we have $\xi\notin T$ and thus $\varphi$ is unramified (Corollary \ref{CorUnram}). Furthermore $\mathcal{W}$ is integral and geometrically unibranched being an irreducible component of the spectrum of a reduced complete local algebra over an algebraically closed field. Therefore, \cite[Th\'eor\`eme 18.10.1]{EGAIV} implies that $\varphi$ is \'etale. Since $\mathcal{W}'$ is smooth, it follows that $\mathcal{W}$ is smooth as well.
\par On the other hand, the component $\Xi$ is singular at $[Y]$ by \cite[p. 59]{Li-Oort}. (The reason is that the map $\phi_\text{LO}$ contracts a curve of positive genus to the point $[Y]$, see also Theorem \ref{ThmLiOortMap}.)
\end{proof}
\section{Non-transversality criteria}
\subsection{General case}
Let $C$ be a hyperelliptic supersingular curve of genus $3$. Let $\mathcal{W}$ be an irreducible component of the completion of $\mathcal{S}_3$ at the moduli point $[C]$. Assume that $\mathcal{W}$ is smooth at $[C]$. In this section we derive a necessary and sufficient criterion for the non-transversality of the intersection of $\mathcal{W}$ and $\mathfrak{H}_3$ at $[C]$. We first describe the criterion geometrically. Later we specialize to the case $a=1$, where the criterion becomes equivalent to a condition on the Cartier-Manin matrix.
\par We begin with a lemma on PFTQs:
\begin{lem}\label{LemPFTQPull}
Let $E^3\longrightarrow Y_0$ be a PFTQ (we omit $Y_1$ from the notation). Then
\begin{enumerate}
\item[i)] $E^3\rightarrow Y_0$ factors through multiplication by $p$. Denote by
$$\psi: Y_0\longrightarrow E^3$$
the factored map.
\item[ii)] The images of the two maps
$$\psi^*: \HH^0(E^3, \Omega_{E^3})\longrightarrow \HH^0(Y_0, \Omega_{Y_0})$$
$$\psi^*: \HH^1(E^3, \mathcal{O}_{E^3})\longrightarrow \HH^1(Y_0, \mathcal{O}_{Y_0})$$
are both $1$-dimensional, unless $\psi=F$. (In the latter case they are both $0$.)
\item[iii)] Furthermore, the two images in ii) are orthogonal with respect to the natural pairing induced by the principal polarization on $Y_0$.
\end{enumerate}
\end{lem}
\begin{proof}
The first assertion follows from the fact that the polarization $\eta$ in Definition \ref{DefPFTQ} satisfies $\ker(\eta)=E^3[p]$. For the assertion ii) we use the explicit description of the Dieudonn\'e modules. Recall that $M_2$ was the Dieudon\'e of $E^3$ as discussed in Section \ref{SecDieu}. We will denote by $M_0$ the Dieudonn\'e module of $Y_0$ viewed as a submodule of $M_2$ via the map induced by $E^3\rightarrow Y_0$.
\par Then the map $\psi: Y_0\rightarrow E^3$ induces on Dieudonn\'e modules the inclusion $p M_2 \subset M_0$. The map
$$\psi^*: \HH^0(E^3, \Omega_{E^3})\longrightarrow \HH^0(Y_0, \Omega_{Y_0})$$
is identified with
$$\frac{V (pM_2)}{p(pM_2)  }\rightarrow \frac{V M_0}{pM_0}$$
under the isomorphism from \cite[Corollary 5.11]{Oda}.
From the explicit description of the Dieudonn\'e modules in Section \ref{SecDieu} (which uses the assumption that $\psi \neq F$) one readily sees that the image of $\psi^*$ is one-dimensional and the basis corresponds to the element $g_3=pFm_0$. Similarly the map 
$$\psi^*: \HH^1(E^3, \mathcal{O}_{E^3})\longrightarrow \HH^1(Y_0, \mathcal{O}_{Y_0})$$
equals the map
$$\frac{(pM_2)}{V(pM_2)  }\rightarrow \frac{M_0}{VM_0}\,.$$
whose image has basis $c_1=p m_0$. This shows ii).
\par For iii) we need to make the pairing between $\HH^0(Y_0, \Omega_{Y_0})$ and $\HH^1(Y_0, \mathcal{O}_{Y_0})$ explicit. Indeed, the polarization on $\eta$ induces an alternating non-degenerate pairing
$$M_2^t \times M_2^t \longrightarrow W(k)\,.$$
This gives an isomorphism $M_2^t\!\left[\frac{1}{p} \right] \stackrel{\sim}{\longrightarrow} M_2\!\left[\frac{1}{p} \right]$. Therefore, we get a pairing
$$M_2 \times M_2 \longrightarrow W(k)\!\left[\frac{1}{p} \right]\,.$$
The restriction to $M_0$ again takes values in $W(k)$ because $\eta$ descends to a principal polarization on $Y_0$. Thus we get a pairing
$$\langle\cdot, \cdot \rangle:M_0\times M_0 \longrightarrow W(k)\,.$$
From this and the explicit description of the pairing $\langle \cdot , \cdot \rangle$ in Section \ref{SecDieu} one computes $\langle p m_0, pFm_0\rangle=p$. Claim iii) follows.
\end{proof}
Now let $C$ be as at the start of this chapter. The choice of an irreducible component $\mathcal{W}$ of the completion of $\mathcal{S}_3$ at the moduli point $[C]$ is equivalent to the choice of a PFTQ ending in $\Jac(C)$ by Lemma \ref{LemIrredComp}. This gives us a map
$$\psi: \Jac(C) \longrightarrow E^3$$
as in Lemma \ref{LemPFTQPull}. Recall that we assumed that $\mathcal{W}$ was non-singular at $[C]$, which by Corollary \ref{CorIrredComplSm} is equivalent to $\psi\neq F$. 
\begin{definition}\label{DefFil}
We define the following filtration on $V=\HH^0(C, \Omega_C)$:
\begin{itemize}
\item $V_0=V$.
\item $V_1$ is defined to be the orthogonal complement of
$$\im\! \left(\psi^*: \HH^1(E^3, \mathcal{O}_{E^3})\longrightarrow \HH^1(\Jac(C), \mathcal{O}_{\Jac(C)})\cong \HH^1(C, \mathcal{O}_{C}) \right)$$
with respect to the Serre duality pairing.
\item $V_2$ is defined to be
$$V_2=\im\! \left(\psi^*: \HH^0(E^3, \Omega_{E^3})\longrightarrow \HH^0(\Jac(C), \Omega_{\Jac(C)}) \cong \HH^0(C, \Omega_C)\right)$$
\end{itemize}
\end{definition}
Notice that Lemma \ref{LemPFTQPull} implies that $V_2\subset V_1$ and $\dim(V_i)=3-i,\, i=0,1,2$.
\par Consider now the canonical map $\varphi_{\text{can}} : C \longrightarrow \mathbb{P}(\HH^0(C, \Omega_C))\cong \mathbb{P}^2$. The image of $\varphi_{\text{can}}$ is a conic because $C$ is hyperelliptic.
\par On the other hand, the filtration $V_2\subset V_1\subset V_0=\HH^0(C, \Omega_C)$ defines a point $P$ and a line $l$ containing $P$ in $ \mathbb{P}(\HH^0(C, \Omega_C))$:
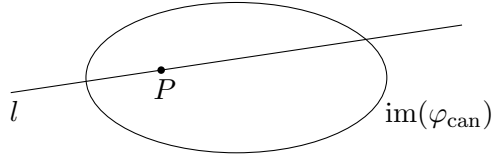
\begin{figure}[H]\label{FigTouch}
\begin{tikzpicture}
\draw (0,0) ellipse (2cm and 1cm);
\node at (2.7,-0.5) {$\im(\varphi_{\text{can}})$};
\fill (-1,0.1) circle (0.05cm and 0.05cm);
\node at (-1,0.1) [label={[shift={(0.05,-0.65)}] $P$}] {};
\draw (-3, -0.2) -- (3, 0.7);
\node at (-3, -0.2)[label={[shift={(0.05,-0.65)}] $l$}] {};
\end{tikzpicture}
\caption{Generic picture}
\end{figure}
We are now ready to state Theorem A from the introduction. Let $\mathcal{W}$ be as above. By Lemma \ref{LemIrredComp} there is a PFTQ $\pi:E^3 \rightarrow \Jac(C)$ corresponding to $\mathcal{W}$. Furthermore the smoothness assumption on $\mathcal{W}$ is equivalent to saying that $\pi \neq F$. We can thus define the line $l$ and the point $P$ via the previous definition. 
\begin{thm}\label{ThmTransGen}
The following are equivalent:
\begin{enumerate}
\item[i)] The component $\mathcal{W}$ meets $\mathfrak{H}_3$ non-transversally at $[C]$.
\item[ii)] The line $l$ touches $\im(\varphi_{\text{can}})$ at the point $P$.
\end{enumerate}
\end{thm}
\begin{proof}
The tangent space of $\mathcal{W}$ at the point $[C]$ is two-dimensional with basis given by the two Kodaira-Spencer classes $\kappa(\delta_1),\kappa(\delta_2)$ as in Lemma \ref{LemKS}. We can interpret  $\kappa(\delta_1),\kappa(\delta_2)$ as elements in $S^2 \HH^0(C, \Omega_C)^\vee$. By Lemma \ref{LemHypDefo}, the component $\mathcal{W}$ meets $\mathfrak{H}_3$ non-transversally at $[C]$ if and only if $\kappa(\delta_1),\kappa(\delta_2)$ vanish on the kernel of
$$S^2 \HH^0(C, \Omega_C) \longrightarrow \HH^0(C, \Omega_C^{\otimes 2})\,.$$
Since $g=3$ this kernel is 1-dimensional and generated by the quadratic form cutting out the conic $\im(\varphi_{\text{can}})$.
\par Thus we need to show that the orthogonality of this quadratic form to the Kodaira-Spencer classes $\kappa(\delta_1),\kappa(\delta_2)$ has the desired geometric interpretation.
\par Indeed, recall that we computed in Lemma \ref{LemKS}
$$\kappa(\delta_1)(g_1)=-p m_0=-c_1,\qquad \kappa(\delta_1)(g_2)=\kappa(\delta_1)(g_3)=0$$
$$\kappa(\delta_2)(g_1)\in \lspan(c_1, c_2) ,\qquad \kappa(\delta_2)(g_2)=\tx_2^p pm_0,\qquad\kappa(\delta_2)(g_3)=0\,$$
under the additional assumption that $x_2$ is not zero at $\xi$ which we can assume without loss of generality.
\par We will now compute the Serre duality pairing
$$\langle \cdot, \cdot \rangle : \HH^0(C, \Omega_C) \times \HH^1(C, \mathcal{O}_C) \longrightarrow k$$
on some of the given basis elements $g_1, g_2, g_3$ resp. $c_1, c_2, c_3$. Indeed this pairing equals the pairing induced by the principal polarization on ${Y_0\cong \Jac(C)  }$. This has the explicit description given in the proof of Lemma \ref{LemPFTQPull}.
\par Thus we can compute
$$\langle g_2, c_1 \rangle=\langle g_3, c_1 \rangle=\langle g_3, c_2 \rangle = 0\,. $$
This implies that $\lspan(c_1)^\perp=\lspan(g_2, g_3)$. But that means that $\kappa(\delta_1) $ is the unique map (up to scalars)
$$  \HH^0(C, \Omega_C) \longrightarrow \HH^1(C, \mathcal{O}_C)$$
that vanishes on the two-dimensional vector space $\lspan(c_1)^\perp$ and has image $\lspan(c_1)$.
\par It is not difficult to see that this means that the quadratic forms in $S^2 \HH^0(C, \Omega_C)$ orthogonal to $\kappa(\delta_1)$ are exactly those vanishing at the point in
$$\mathbb{P}(\HH^0(C, \Omega_C))$$
corresponding to $\lspan(c_1)$. But that point is equal to $P$ by its definition. Therefore, we conclude that $\kappa(\delta_1)$ represents a tangent direction to the hyperelliptic locus if and only if $P\in \im(\varphi_{\text{can}})$.
\par Next, we analyze $\kappa(\delta_2)$. Indeed, we showed above
$$\lspan(c_1, c_2)^\perp=\lspan(g_3)\,.$$
Therefore
\begin{align}
\im\!\left(\kappa\left(\delta_2\right)\right)\subseteq \lspan\!\left(c_1, c_2\right)\,,\\
\kappa\left(\delta_2\right)\! \left(\lspan\!\left(c_1\right)^\perp\right)\subseteq \lspan\!\left(c_1\right)\,,\\
\kappa\left(\delta_2\right)\! \left(\lspan\!\left(c_1,c_2\right)^\perp\right)=0\,. 
\end{align}
Let us denote by $\Upsilon\subset \Hom^\text{sym}(\HH^0(C, \Omega_C )\longrightarrow \HH^1(C, \mathcal{O}_C)  )$ the linear subspace of maps satisfying the conditions (1)-(3) from above. One has ${\dim(\Upsilon)=2}$ because $\Upsilon$ is equal to the vector space of maps whose matrix is of the form
$$X=\begin{pmatrix}
\ast & \ast & 0\\
\ast & 0 & 0\\
0 & 0 & 0
\end{pmatrix},\, X=X^T$$
with respect to the basis $c_1, c_2, c_3$ on $\HH^1(C, \mathcal{O}_C)$ and its dual basis on $\HH^0(C, \Omega_C)$.
\par But we also have $\kappa(\delta_1)\in \Upsilon$. Since $\kappa(\delta_1), \kappa(\delta_2)$ are linearly independent, we conclude that $\Upsilon=\lspan(\kappa(\delta_1), \kappa(\delta_2))$.
\par Now let $Q\in S^2 \HH^0(C, \Omega_C)$ be arbitrary. Then $Q$ is contained in the orthogonal complement of $\Upsilon$ if and only if $\mathcal{V}(Q)$ is a conic touching the line $\mathcal{V}(g_3) $ at the point $P$. (This can be seen from the matrix representation above.) But that line equals $l$ by definition.
\par Thus we may conclude: The component $\mathcal{W}$ meets $\mathfrak{H}_3$ non-transversally at $[C]$ if and only of the line $l$ touches $\im(\varphi_{\text{can}})$ at the point $P$. The theorem follows.
\end{proof}
\subsection{Simplifications for $a=1$}
We will now specialize to the case $a=1$. By \cite[Remark on p. 38]{Li-Oort} this is equivalent to saying that there is a unique component of $\mathcal{S}_3$ passing through our moduli point.
\par Our first lemma will describe the filtration $V_2\subset V_1\subset V=\HH^0(C, \Omega_C)$ in terms of the Cartier operator
$$\mathcal{C}: \HH^0(C, \Omega_C)\longrightarrow \HH^0(C, \Omega_C)$$
as follows:
\begin{lem}\label{LemCarta1}
Assume further that $a=1$. Let $E^3\longrightarrow \Jac(C)$ be the unique PFTQ ending at $\Jac(C)$. Let further $V_2\subset V_1\subset V=\HH^0(C, \Omega_C)$ be as in Definition \ref{DefFil}.
\par Then $V_2=\ker(\mathcal{C})$ and $V_1=\ker\! \left(\mathcal{C}^2\right)$.
\end{lem}
\begin{proof}
The assumption $a=1$ implies that $\dim(\ker(\mathcal{C}))=1$. Therefore, it suffices to show that $V_2\subseteq \ker(\mathcal{C})$. This follows from
$$V_2=\im\! \left(\pi^*: \HH^0(E^3, \Omega_{E^3})\longrightarrow \HH^0(\Jac(C), \Omega_{\Jac(C)})\cong \HH^0(C, \Omega_C) \right)$$
and the naturality of the Cartier operator.
\par To show that $V_1=\ker\! \left(\mathcal{C}^2\right)$ it suffices again to prove $V_1\subseteq\ker(\mathcal{C}^2)$. Recall that $V_1$ is defined to be the orthogonal complement of
$$\im\! \left(\pi^*: \HH^1(E^3, \mathcal{O}_{E^3})\longrightarrow \HH^1(\Jac(C), \mathcal{O}_{\Jac(C)})\cong \HH^1(C, \mathcal{O}_{C}) \right)$$
with respect to the Serre duality pairing. But since $F$ vanishes on this image, the identity
$$\langle Fx, y\rangle = \langle x, \mathcal{C}(y) \rangle^p $$
shows that $V_1$ is closed under $\mathcal{C}$. Now $\mathcal{C}$ is nilpotent as a consequence of the supersingularity of $C$. Therefore,
$$V_1\subseteq\ker\! \left(\mathcal{C}^{\dim(V_1)}\right)=\ker\! \left(\mathcal{C}^2\right)\,.$$
This proves the lemma.
\end{proof}
We can now state a reformulation of Theorem \ref{ThmTransGen} in the special case $a=1$ leading to Theorem B from the introduction:
\begin{thm}\label{ThmTransa1}
Let $C$ be a hyperelliptic supersingular curve of genus $3$ with $a(\Jac(C))=1$. Denote the hyperelliptic involution by $\tau: C\longrightarrow C$. Then the following are equivalent:
\begin{enumerate}
\item[i)] The supersingular locus $\mathcal{S}_3$ meets $\mathfrak{H}_3$ non-transversally at $[C]$.
\item[ii)] There exists a point $P\in C$ such that the filtration $V_2\subset V_1\subset V_0=\HH^0(C, \Omega_C)$ of Lemma \ref{LemCarta1} agrees with the filtration
$$\HH^0\left(C, \Omega_C(-2 P-2 \tau(P) \right)\subset \HH^0\left(C, \Omega_C(- P-\tau(P) \right)\subset \HH^0\left(C, \Omega_C \right)$$
\item[iii)] There exists a point $P\in C$ such that for one (equivalently: for all) hyperelliptic equation
$$y^2=f(x)$$
of $C$ that satisfies $x(P)=0$ the corresponding Cartier-Manin matrix (formed by extracting suitable coefficients of $f^\frac{p-1}{2}$) has the shape
$$\begin{pmatrix}
0 & 0 & 0\\
\ast & 0 & 0\\
\ast & \ast & 0
\end{pmatrix}$$
\end{enumerate}
\end{thm}
\begin{proof}
In i) we use that $a=1$ implies that there is unique component of $\mathcal{S}_3$ passing through $[C]$ \cite[p. 38]{Li-Oort} which furthermore is smooth at $[C]$ (Corollary \ref{CorUnram}). The implication ``i)$\Leftrightarrow$ii)'' is a reformulation of Theorem \ref{ThmTransGen}. 
\par Assume now that ii) is satisfied. We will show that iii) holds true with the same point $P\in C$. Indeed let
$$y^2=f(x)$$
be an arbitrary hyperelliptic equation for $C$ satisfying $x(P)=0$. Then the filtration 
$$\HH^0\left(C, \Omega_C(-2 P-2 \tau(P) \right)\subset \HH^0\left(C, \Omega_C(- P-\tau(P) \right)\subset \HH^0\left(C, \Omega_C \right)$$
is explicitly given by
$$\lspan\! \left(x^2 \frac{dx}{y}\right) \subset \lspan\! \left(x \frac{dx}{y}, x^2 \frac{dx}{y}\right)\subset \lspan\! \left(\frac{dx}{y}, x \frac{dx}{y}, x^2 \frac{dx}{y}\right)\,.$$
By assumption and Lemma \ref{LemCarta1} one then has
$$\lspan\! \left(x^2 \frac{dx}{y}\right)=\ker(\mathcal{C}),\, \lspan\! \left(x \frac{dx}{y}, x^2 \frac{dx}{y}\right)=\ker\! \left(\mathcal{C}^2\right)\,.$$
Together with the nilpotence of $\mathcal{C}$ this implies that the matrix of $\mathcal{C}$ with respect to the basis $\frac{dx}{y}, x \frac{dx}{y}, x^2 \frac{dx}{y}$, i.e., the Cartier-Manin matrix, has the shape
$$\begin{pmatrix}
0 & 0 & 0\\
\ast & 0 & 0\\
\ast & \ast & 0
\end{pmatrix}$$
proving iii).
\par The implication iii)$\Rightarrow$ ii) is shown by reversing the previous argument. This concludes the proof of the theorem.
\end{proof}
The utility of part iii) of the Theorem is the fact that it is easy to check whether the criterion is satisfied for a given hyperelliptic equation $y^2=f(x)$.
\begin{definition}
\begin{itemize}
\item Let $C$ be a hyperelliptic curve of genus $3$ with hyperelliptic involution $\tau$ and $P\in C(k)$ be a point. We define the \emph{filtration associated to $P$} to be the filtration in Theorem \ref{ThmTransa1} ii):
$$\HH^0\left(C, \Omega_C(-2 P-2 \tau(P) \right)\subset \HH^0\left(C, \Omega_C(- P-\tau(P) \right)\subset \HH^0\left(C, \Omega_C \right)\,.$$
\item When $a(\Jac(C))=1$ and one (and hence both) of conditions ii) or iii) of Theorem \ref{ThmTransa1} are satisfied, we call $P$ a \emph{touchpoint for the Cartier operator}.
\end{itemize}
\end{definition}
A touchpoint for the Cartier operator (if it exists) is unique up to replacing $P$ by $\tau(P)$.
\section{Examples}
\subsection{Na\"\i ve method ($a=1$)}
Using the na\"\i ve method described in the introduction one finds the following examples of supersingular hyperelliptic curves with $a$-number $1$ satisfying the criterion of Theorem \ref{ThmTransa1}.  For all these curves the Cartier-Manin matrix has the shape
$$\begin{pmatrix}
0 & 0 & 0\\
\ast & 0 & 0\\
\ast & \ast & 0
\end{pmatrix}\,.$$
The list is not proven to be complete; for fixed $p$ there could be more examples not isomorphic to the ones in the table.\\
\bgroup
\def\arraystretch{1.5}
\begin{tabular}{c|c}
$p$ & $C$\\
\hline
\hline
$11$ & $y^2=x^7 + x^5 + 7x^3 + x$\\
\hline
$19$ & $y^2 = x^7 + 7x^5 + 14x^3 + 7x$\\
     & $y^2=x^7 + 4\sqrt{3} x^5 + (12 \sqrt{3} + 1)x^3 + x$\\
     & $y^2 = x^7 + x^6 + 9x^5 + 5x^4 + 6x^3 + x^2 + 16x + 1$\\
\hline
$23$ & $y^2 = x^7 + 7x^5 + 14x^3 + 7x$\\
\hline
$31$ & $y^2 = x^7 + (\sqrt{3} + 6) x^5 + (8\sqrt{3} + 23)x^3 + 17x$\\
\hline
$43$ & $y^2= x^7 + (3b^2 + 25b + 40)x^5 + (28b^2 + 10b + 7)x^3 + x,\, b=\sqrt[3]{3}$\\
\hline
$47$ & $y^2 = x^7 + 7x^5 + 14x^3 + 7x$\\
\hline
$59$ & $y^2= x^7 + (32b^2 + b + 16)x^5 + (28b^2 + 43b + 1)x^3 + 7x,$\\
    & where $b\in \mathbb{F}_{59^3}$ is a solution of $b^3 + 5b  -2=0$\\
\hline
$67$ & $y^2 = x^7 + 18x^5 + 3x^3 + x$\\
     & $y^2 = x^7 + (33 b^3 + 43b^2 + 64b + 6)x^5 + (54b^3 + 8b^2 + 26b + 7)x^3 + x,$\\
     & where $b\in \mathbb{F}_{67^4}$ is a solution of $b^4 + 8b^2 + 54b + 2=0$ \\     
\hline
$71$ & $y^2= x^7 + 7x^5 + 17x^3 + 5x$\\
     & $y^2= x^7 + 2x^5 - 6x^3 + 21x$\\
\hline
$79$ & $y^2 = x^7 + 7x^5 + 14x^3 + 7x$\\
     & $y^2 = x^7 + (9\sqrt{3} + 5)x^5 + (8\sqrt{3} + 15)x^3 + 2x$
\end{tabular}
\egroup
\vspace{1cm}
\par The curve $y^2=x^7 + 7x^5 + 14x^3 + 7x$ appears several times in this list. As we shall see in the next subsection this can be explained with CM-theory.

\begin{rem}
The alert reader will notice that all the primes in this list satisfy $p\equiv 3 \mod 4$. Furthermore, all curves, except $y^2 = x^7 + x^6 + 9x^5 + 5x^4 + 6x^3 + x^2 + 16x + 1$ have an automorphism of order $4$ given by $x\mapsto -x, \, y\mapsto \sqrt{-1} y$. This has the following explanation: Indeed, consider the Picard surface $S\subseteq \mathfrak{A}_3$ which parametrizes principally polarized abelian threefolds $A$ together with an embedding $\mathbb{Z}[i] \hookrightarrow \End(A)$ such that the action of $i$ on $T_0 A$ has eigenvalues $(\sqrt{-1},\sqrt{-1},-\sqrt{-1})$. It follows from \cite[Theorem 4.5]{Weng} that $S$ is generically contained in the hyperelliptic locus. The corresponding hyperelliptic curves will have an equation of the form
\begin{equation}\label{eq:aut}\tag{$\ast$}
y^2 = a_7 x^7+a_5 x^5+a_3 x^3+a_1 x
\end{equation}
Now, if $p\equiv 3 \mod 4$, then $p$ is inert in $\mathbb{Z}[i]$ and hence by \cite[Theorem 2.1]{GorenShalit}, the supersingular locus $S_\text{ss}$ of $S$ is equidimensional of dimension $1$. Furthermore, the irreducible components of $S_\text{ss}$ are smooth and intersect in points with $a=3$. Let $C$ be one of the curves in the table above with an equation of the form $(\ast)$. Since the irreducible componets of $S_\text{ss}$ are smooth, there must be an additional irreducible component $Z$ of $\mathcal{S}_3 \cap \mathfrak{H}_3$ passing through the moduli point $[\Jac(C)]$. Furthermore $Z$ cannot be contained in $S_\text{ss}$ because $a(\Jac(C))=1$. We conclude that the non-transversality of $\mathcal{S}_3$ and $\mathfrak{H}_3$ is explained by the existence of two types of irreducible components of the intersection: Those coming from $S_\text{ss}$ and those that do not. The two types of irreducible components can intersect anywhere in $\mathcal{S}_3$, even at points with $a$-number $1$.
\end{rem}
\subsection{CM-examples ($a=1$)}
In this section we will construct an example of a curve $C$ over $\mathbb{Q}$ with the property that the reduction at any prime $p$ with $p\equiv \pm 2 \mod 7, p\equiv 3 \mod 4$ satisfies the criterion of Theorem \ref{ThmTransa1}.
\par To set up our notation consider a hyperelliptic genus $3$ curve $C$ over $\mathbb{C}$. Suppose that $\Jac(C)$ has CM by a field $K$. Let us look at the differential forms $V=\HH^0(C, \Omega_C)$. By CM-theory there is a decomposition
$$V=\bigoplus_{\iota \in \Phi^c} V_{\iota} $$
into one dimensional $\mathbb{C}$-vector spaces. Here $\Phi\subset \Hom(K, \mathbb{C})$ is the CM-type and $\Phi^c$ its complement. An element $a\in K$ acts on $V_{\iota}$ via multiplication by $\iota(a)$.
\par Furthermore we denote by $L$ the Galois closure of $K$ over $\mathbb{Q}$ and by $G=\Gal(L/\mathbb{Q})$ the Galois group.
\begin{definition}\label{DefCMadapt}
Let $C$ be as above and let $P\in C(\mathbb{C})$ be a point. We say that \emph{$P$ is a touchpoint for the CM-action} if the following condition holds true:
\begin{itemize}
\item[i)] Denote by $V=V_0 \supset V_1 \supset V_2$ the filtration associated to $P$. Then we demand that there is an ordering $\iota_0, \iota_1, \iota_2$ of $\Phi^c$ such that
$$V_{\iota_i} \subseteq V_i \text{ for all } i \in \{0,1,2\} \,.$$
\item[ii)] The choice of the ordering in i) gives an ordering $\iota_0,\iota_1,\iota_2,\overline{\iota_0},\overline{\iota_1},\overline{\iota_2}$ of all the embeddings of $K$ into $\mathbb{C}$. This gives a map
$$\text{per}:G\hookrightarrow S_6$$
induced from the action of $G$ on the complex embeddings of $K$. The second condition is that
$$(123456)\in \im(\text{per})\,.$$
\end{itemize}
\end{definition}
\begin{rem} The motivation for the terminology is that we can draw the following picture: From the decomposition $V=\bigoplus_{\iota \in \Phi^c} V_{\iota} $ we get three points $P_0, P_1, P_2$ in $\mathbb{P}^2$, namely $P_i$ is the point corresponding to the surjection $\pr_i: \bigoplus_{\iota \in \Phi^c} V_{\iota} \rightarrow V_{\iota_i}$ for $i\in \{0,1,2\}$. Then condition i) is equivalent to $P_0=\varphi_{\text{can}}(P)$ and the line $P_0P_1$ should be tangent to $\im(\varphi_{\text{can}})$. Condition ii) does not have a geometric interpretation.
\begin{figure}[H]\label{FigTouchCM}
\begin{tikzpicture}
\draw (0,0) ellipse (2cm and 1cm);
\node at (2.7,-0.5) {$\im(\varphi_{\text{can}})$};
\fill (0,1) circle (0.05cm and 0.05cm);
\node at (0,1) [label={[shift={(0.1,-0.65)}] $P_0$}] {};
\draw (-3, 1) -- (3, 1);
\fill (2,1) circle (0.05cm and 0.05cm);
\node at (2, 1)[label={[shift={(0.1,-0.65)}] $P_1$}] {};
\fill (-2.5,-0.5) circle (0.05cm and 0.05cm);
\node at (-2.5, -0.5)[label={[shift={(0.1,-0.65)}] $P_2$}] {};
\end{tikzpicture}
\caption{CM touchpoint}
\end{figure}
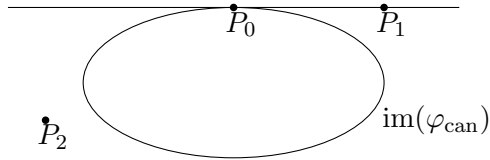

\par Notice that the existence of a touchpoint $P$ for the CM-action is very restrictive. Most of the CM-curves do not have any. Contrary to touchpoints for the Cartier operator there could exist more than one touchpoint for the CM-action (modulo the hyperelliptic involution) because it might happen that $P_2$ is on $\im(\varphi_{\text{can}})$ with tangent line $P_1 P_2$.
\par A touchpoint can be a Weierstrass point. We will now discuss an example where this happens.
\end{rem}
\begin{exmpl}\label{ExmplTouch}
Consider the hyperelliptic curve $C_{14}$ over $\mathbb{C}$ given by the affine equation
$$y^2=x\left(x^{14}-1\right)\,.$$
It is known \cite{BrandtStich} that the automorphism group of $C_{14}$ is isomorphic to the group $U_{14}$, where
$$U_{n}=\langle a,b \mid a^2=b^{2n}=abab^{n+1} =1\rangle \,.$$
The action is given by
$$a:\begin{array}{c}
        x\mapsto \frac{-1}{x}\\
        y \mapsto \frac{y}{x^8}
\end{array},\, b:\begin{array}{c}
        x\mapsto -\zeta_7 x\\
        y \mapsto i \zeta_7^4\, y
\end{array}$$
where $\zeta_7=e^{\frac{2 \pi i}{7}}\in \mathbb{C}$.
\par We now define $C$ to be the quotient $C_{14}/a$. Denote by
$$\pi:C_{14}\longrightarrow C$$
the quotient map. We define a point $P$ on $C$ given by the image of\\
$(0,0)\in C_{14}$ under $\pi$.
\end{exmpl}
\begin{rem}
The curve $C$ was first found by Annegret Weng. It is the first entry of the list \cite[p. 18]{Weng} and has the following Weierstrass equation:
$$y^2=x^7+7x^5+14x^3+7x$$
Together with Bert van Geemen and Kenji Koike \cite[Example 2.1]{GeeKoiWe} she used the construction described in the example to prove that $\Jac(C)$ has CM by $\mathbb{Q}(\zeta_7+\zeta_7^{-1},\, i)$.
\par The construction for the curve $C$ can be seen as a special case of Tautz-Top-Verberkmoes \cite{TTV} in the case where (in their notation) $p=7,\, t=0$.
\end{rem}
To facilitate reading the proof of the following proposition, we remark that we do not choose coordinates on $C$. Thus the coordinates $x,y$ always refer to $C_{14}$.
\begin{prop}\label{PropTouch} In the notation from above:
\begin{enumerate}
\item[i)] The curve $C$ is hyperelliptic of genus $3$.
\item[ii)] $\Jac(C)$ has CM by $K=\mathbb{Q}(\zeta_7+\zeta_7^{-1},\, i)$. The CM-order
$$R=\End(\Jac(C)) \cap K$$
is maximal away from $2$. (In fact $R=\mathcal{O}_K$, but we will not use this.)
\item[iii)] The point $P$ is a touchpoint for the CM action.
\end{enumerate}
\end{prop}
\begin{proof}
For i) notice that the automorphism $a$ of $C_{14}$ has $4$ fixed points, namely the points with $x$-coordinate $\pm i$. Then the Riemann-Hurwitz formula shows that $g(C)=3$. Moreover, $C$ is hyperelliptic because every curve with a finite map from a hyperelliptic curve is either (hyper)-elliptic or has genus $0$ (look at the push-forward of the divisor defining the hyperelliptic pencil).
\par We will now show ii) and iii). Indeed, consider the action of the group $U_{14}$ on $C_{14}$. It induces a map
$$\varphi:\mathbb{Z}\!\left[\frac{1}{2}\right][U_{14}]\longrightarrow \End(\Jac(C_{14}))\!\left[\frac{1}{2}\right]\,.$$
Consider the two elements 
$$\mathsf{i}=\varphi\left(b^7\right),\, \alpha=\varphi\! \left(\frac{b^8+b^{-8}+a b^8+ a b^{-8}}{2}\right)\in \End(\Jac(C_{14}))\!\left[\frac{1}{2}\right]\,.$$
We claim that $\mathsf{i}, \alpha$ give endomorphisms in $\End(\Jac(C))\!\left[\frac{1}{2}\right]$ by restricting them to the image of the map
$$\pi^*: \Jac(C)\hookrightarrow \Jac(C_{14})\,. $$
Indeed, since $\im(\pi^*)=\ker(1-a)^o$ it suffices to point out that
$$a b^7=b^7 a$$
$$a \frac{b^8+b^{-8}+a b^8+ a b^{-8}}{2}= \frac{b^8+b^{-8}+a b^8+ a b^{-8}}{2} a\,.$$
This proves the claim. 
\par It is now elementary to check that the map
$$\mathbb{Z}\!\left[\frac{1}{2}, \zeta_7+\zeta_7^{-1}, \, i \right] \longrightarrow \End(\Jac(C))\!\left[\frac{1}{2}\right],\, \zeta_7+\zeta_7^{-1} \mapsto \alpha,\, i \mapsto \mathsf{i}$$
is a well-defined ring homomorphism. Since $\mathcal{O}_K=\mathbb{Z}[\zeta_7+\zeta_7^{-1}, \, i]$, ii) follows.
\par To show iii) we need to compute the action of $K$ on the differential forms on $C$. To this end, consider the action of the group $U_{14}$ on $\HH^0(C_{14}, \Omega_{C_{14}}) $. Indeed, the vector space $\HH^0(C_{14}, \Omega_{C_{14}})$ has 

 basis $x^j \frac{dx}{y}, j=0,\ldots, 6$. With respect to this basis the automorphisms $a$ and $b$ act via the matrices
$$\begin{pmatrix}
&&&&&&1\\
&&&&&-1\\
&&&& 1\\
&&&-1\\
&& 1\\
& -1\\
1
\end{pmatrix} \text{resp}. \, \begin{pmatrix}
i \zeta_7^4\\
 & (i \zeta_7^4)^3\\
 && \ddots\\
 &&& (i \zeta_7^4)^{13}
\end{pmatrix}\,.$$
This implies that the image of
$$\pi^*: \HH^0(C, \Omega_C) \longrightarrow \HH^0(C_{14}, \Omega_{C_{14}}) $$
has basis
$$\omega_0=\left(1+x^6\right) \frac{dx}{y},\, \omega_1=\left(x-x^5\right) \frac{dx}{y}, \omega_2=\left(x^2+x^4\right) \frac{dx}{y}\,$$
(as it is given by the differential forms fixed by $a$). Moreover, we compute
$$\mathsf{i}^*\left( \omega_0 \right) = -i \omega_0,\, \alpha^*\left( \omega_0 \right) = (\zeta_7^3+\zeta_7^{-3}) \omega_0\,,$$
$$\mathsf{i}^*\left( \omega_1 \right) = i \omega_1, \, \alpha^*\left( \omega_1 \right) = (\zeta_7^2+\zeta_7^{-2}) \omega_1\,, $$
$$\mathsf{i}^* \left(\omega_2 \right) = -i \omega_2, \, \alpha^*\left( \omega_2 \right) = (\zeta_7+\zeta_7^{-1})  \omega_2\,. $$
This gives us the desired description of the action of $K$ on the differential forms on $C$. We get that the point $P=\pi\left(\left(0,0\right)\right)$ is a touchpoint for the CM-action.
\end{proof}
\begin{rem}
\par The idea behind the choice of the field $\mathbb{Q}(\zeta_7+\zeta_7^{-1},i)$ is as follows: The well-known CM curve $y^2 = x^7+1$ has CM by the cyclotomic field $\mathbb{Q}(\zeta_7)$ and satisfies condition i) in Definition \ref{DefCMadapt} either with $P=(0,1)$ or with $P=\infty$. However, both of the points fail condition ii).
\par For that reason we will look for a hyperelliptic curve with CM by another field. The field $\mathbb{Q}(\zeta_7+\zeta_7^{-1}, i)$ is the next option (if you order sextic CM fields by discriminant) \cite[Table 1]{Kilicer}).
\end{rem}
For the next theorem we use the following notation: Let $C$ be a hyperelliptic genus $3$ curve over $\mathbb{C}$ such that $\Jac(C)$ has CM by a field $K$. We denote by $L$ the Galois closure of $K/\mathbb{Q}$ and by $G=\Gal(L/\mathbb{Q})$ the Galois group. Assume, furthermore, that there exists a point $P\in C$ which is a touchpoint for the CM-action. Let us denote by
$$\text{per}:G\hookrightarrow S_6$$
the group homomorphism introduced in Definition \ref{DefCMadapt}.ii)
\par By CM-theory $C$ is defined over a number field, say $M$. We assume that $M$ is large enough, such that $M$ contains $L$, all endomorphisms of $\Jac(C)$ are defined over $M$, and such that $P$ is defined over $M$ (notice that $P$ must be defined over $\overline{\mathbb{Q}}$).
\begin{thm}\label{ThmCMRed}
Let $\mathfrak{P}$ be a prime ideal of $M$ such that $C$ has good reduction at $\mathfrak{P}$. Assume that $p$ does not divide $\disc(R)$ where $R=K\cap \End_{\mathbb{C}}(\Jac(C))$ is the CM-order and $p$ is the residue characteristic of $\mathfrak{P}$.
\par Denote by $(\overline{C}, \overline{P})$ the reduction of the pair $(C, P)$ at $\mathfrak{P}$ and by $\mathfrak{p}$ the intersection $\mathfrak{P}\cap L$.
\par Assume that $\text{\emph{per}}(\Frob_{\mathfrak{p}})=(123456)^{-1}$.
\par Then the curve $\overline{C}$ is supersingular and the pair ($\overline{C}, \overline{P}$) satisfies the criterion of Theorem \ref{ThmTransa1}, i.e., $\overline{P}$ is a touchpoint for the Cartier operator on $\overline{C}$.
\end{thm}
\begin{proof}
In order to prove the theorem we need to describe the action of the Cartier operator on $\HH^0(\overline{C}, \Omega_{\overline{C}})$ in terms of CM-theory. By \cite[Corollary 5.11]{Oda} it suffices for that purpose to understand the Verschiebung operator on the Dieudonn\'e module $M=D(\Jac(\overline{C}))$. Now the description of the Dieudonn\'e module of the reduction of a CM abelian variety is known (see e.g. \cite{Deligne} cited after \cite[pp. 83-85]{Milne} or look at \cite{Anschuetz} for a stronger result in $p$-adic Hodge theory). Nevertheless, we will reproduce a proof for the special case at hand because we need the inner mechanics of the identification for showing the other claims of the theorem.
\par Indeed, choose $k$, an algebraically closed field containing the residue field of $\mathfrak{P}$. By slight abuse of notation we will denote the base change of $\overline{C}$ to $k$ by the same letter.
\par We claim that the Dieudonn\'e module $M=D(\Jac(\overline{C}))$ has a $W(k)$ basis $m_1, \ldots, m_6$ such that
$$V m_1= m_2, V m_2=m_3, Vm_3=p m_4, V m_4=p m_5, V m_5=p m_6, V m_6=m_1\,.$$
Indeed $M$ naturally carries the structure of an $R\otimes_{\mathbb{Z}} W(k)$-module. By assumption $R\otimes_\mathbb{Z} \mathbb{Z}_p$ is an \'etale $\mathbb{Z}_p$-algebra and therefore
$$R\otimes_{\mathbb{Z}} W(k) \cong \bigoplus_{i=1}^6 W(k)$$
where the copies of $W(k)$ are indexed by $\Hom(R, W(k))$. Now under our assumptions we claim that there is a bijection $\Hom(K, \mathbb{C})\stackrel{\sim}{\longrightarrow} \Hom(R, W(k))$. Indeed, we assumed that $C$ is defined over the number field $M\subset \mathbb{C}$. Furthermore $M$ was assumed to be large enough to contain $L$, a Galois closure of $K$. This gives an embedding $L \hookrightarrow \mathbb{C}$. Now every embedding $\iota: K\hookrightarrow \mathbb{C}$ must factor through $L$. We define the map $\Hom(K, \mathbb{C})\longrightarrow \Hom(R, W(k))$ to be the composition
$$\begin{tikzcd}[cramped, sep=small]
\Hom(K, \mathbb{C}) \arrow{r}& \Hom(K, L) \arrow{r}& \Hom(R, \mathcal{O}_L)\arrow{r}&  \Hom(R, \mathcal{O}_L/\mathfrak{p})\\
& \iota \arrow[mapsto]{r}& \iota_{|R}
\end{tikzcd}$$
followed by the composition $\Hom(R, \mathcal{O}_L/\mathfrak{p}) \rightarrow \Hom(R, k)\rightarrow \Hom(R, W(k)),$ where the first map is induced from the natural inclusion
$$\mathcal{O}_L/\mathfrak{p} \hookrightarrow \mathcal{O}_M/\mathfrak{P}\hookrightarrow k\,,$$
and the second map exists because of the identification $\Hom(R, W(k))\stackrel{\sim}{\rightarrow} \Hom(R\otimes_{\mathbb{Z}}\mathbb{Z}_p, W(k))$ (a consequence of $p$-adic completeness, the assumption that $R\otimes_{\mathbb{Z}}\mathbb{Z}_p$ is an \'etale $\mathbb{Z}_p$-algebra, and general properties of the Witt ring). This gives us the map $\Hom(K, \mathbb{C})\longrightarrow \Hom(R, W(k))$. One can show that this map is a bijection, e.g. by writing down a chain of maps inverting all the arrows.
\par We are now allowed to index the direct sum decomposition 
$$R\otimes_{\mathbb{Z}} W(k) \cong \bigoplus_{i=1}^6 W(k)$$
with the ordering corresponding to the ordering $\iota_0, \iota_1, \iota_2, \overline{\iota}_0, \overline{\iota}_1, \overline{\iota}_2$ of $\Hom(K, \mathbb{C})$.
\par This gives a direct sum decomposition
$$M=\bigoplus_{i=1}^6 M_i$$
as $W(k)$-modules such that $R$ acts on $M_i$ via multiplication by a scalar. More precisely, if $\psi_1, \ldots, \psi_6$ are the maps $R\rightarrow W(k)$ corresponding to $\iota_0, \iota_1, \iota_2, \overline{\iota}_0, \overline{\iota}_1, \overline{\iota}_2$, then $a\in R$ acts on $M_i$ via multiplication by $\psi_i(a)$.
\par Choose at first an arbitrary generator $m_1\in M_1$. (We will later adjust $m_1$.) Then for any $a\in R$ we have
$$a\cdot  Vm_1=V(a\cdot m_1)=V(\psi_1(a) m_1)=(\sigma^{-1}\circ \psi_1)(a) Vm_1$$
where $\sigma: W(k)\rightarrow W(k)$ is the lift of Frobenius. Now the assumption $\text{\emph{per}}(\Frob_{\mathfrak{p}})=(123456)^{-1}$ implies that $\sigma^{-1}\circ \psi_1 = \psi_2$ and thus\\
${m_2 \stackrel{\textrm{def.}}{=} V m_1 \in  M_2}$. We claim that $m_2$ generates $M_2$. For this purpose we look at the isomorphism
$$M/VM \cong \HH^1(\overline{C}, \mathcal{O}_{\overline{C}})\,.$$
Thus $M/VM$ must be generated by the image of $M_4, M_5, M_6$, because the maps $\psi_4, \psi_5, \psi_6$ correspond to the embeddings contained in the CM-type $\Phi$, i.e., $\overline{\iota}_0, \overline{\iota}_1, \overline{\iota}_2$. Therefore $M_2$ is contained in $VM$ and this proves the claim.
\par Similarly we can define $m_3=V m_2$ and $m_3$ must generate $M_3$. Now consider $V m_3$. We have $V m_3\in M_4$. But $V m_3$ does not generate $M_4$ because $M_4/pM_4$ injects into $\HH^1(\overline{C}, \mathcal{O}_{\overline{C}})$ by the argument above. Therefore, $V m_3 \in p M_4$ and we define $m_4=\frac{1}{p} V m_3$. Since this implies that $F m_4=m_3$, we see that $m_4$ must be a generator of $M_4$.
\par Along the same lines of reasoning we obtain generators $m_5$, (resp. $ m_6$) of $M_5$ (resp. $M_6$) satisfying the relations $V m_4=p m_5, V m_5=p m_6$.
\par Then, $V m_6$ will also be a generator of $M_1$, say $V m_6=u m_1$ for some unit $u\in W(k)^\times$. Now, since $k$ is algebraically closed, one can show that there exists a unit $\tilde{u}\in W(k)^\times$ such that $\sigma^{-6}(\tilde{u}) u = \tilde{u}$. After replacing $m_1$ by $\tilde{u} m_1$ and redefining $m_2, \ldots, m_6$ by the relations above, we see that $V m_6=m_1$ holds true. This gives the desired description of the Dieudonn\'e module $M$.
\par Then it follows from this description of $M$ that $\overline{C}$ is supersingular \footnote{Alternatively one can use the Shimura-Taniyama formula to compute the Newton polygon.} and $a(\Jac(\overline{C}))=1$.
\par We claim that $\overline{P}$ is a touchpoint for the Cartier operator. Indeed, by \cite[Corollary 5.11]{Oda} there is an isomorphism
$$\frac{V M}{p M} \cong \HH^0(\overline{C}, \Omega^1_{\overline{C}})$$
that identifies $V$ on the left with the Cartier operator on the right. Therefore, the filtration
$$\ker(\mathcal{C})\subset \ker(\mathcal{C}^2) \subset \HH^0(\overline{C}, \Omega^1_{\overline{C}})$$
is given by
$$\lspan(m_3) \subset \lspan(m_2,m_3) \subset \lspan(m_1, m_2, m_3)\,.$$ 
On the other hand, let us look at the filtration $V_2\subset V_1 \subset V_0$ associated to $\overline{P}$, i.e.
$$\HH^0\left(\overline{C}, \Omega_{\overline{C}}(-2 \overline{P}-2 \tau(\overline{P}) \right)\subset \HH^0\left(\overline{C}, \Omega_{\overline{C}}(- \overline{P}-\tau(\overline{P}) \right)\subset \HH^0\left(\overline{C}, \Omega_{\overline{C}} \right)\,.$$
Definition \ref{DefCMadapt} implies that $V_2$ contains the eigenspace where $R$ acts via the map correponding to the embedding $\iota_2$ (a priori we have such an inclusion on the generic fiber, but it passes to the specialization). Therefore, we see that the class of $m_3$ must be contained in $V_2$ because $m_3$ generates that eigenspace. Similarly we get $m_2 \in V_1$. This shows that $\overline{P}$ is a touchpoint for the Cartier operator on $\overline{C}$.
\end{proof}
\begin{cor}
For every prime number $p$ such that
$$p\equiv \pm 2 \mod 7,\, p \equiv 3 \mod 4$$
the hyperelliptic locus and the supersingular locus intersect non-transversally in $\mathfrak{A}_3\times \mathbb{F}_p$.
\end{cor}
\begin{proof}
The curve $C$ from Example \ref{ExmplTouch} is defined over $\mathbb{Q}$ and has good reduction at every prime different from $2,\, 7$. Furthermore $\Jac(C)$ has CM by $K=\mathbb{Q}(\zeta_7+\zeta_7^{-1}, i)$. Let $R=\End_{\mathbb{C}}(\Jac(C))\cap K$ be the CM-order. By Proposition \ref{PropTouch} the index of $R$ in $\mathcal{O}_K$ must be a power of $2$.
\par Now Proposition \ref{PropTouch} implies that $C$ satisfies the assumptions of Theorem \ref{ThmCMRed}. Let $p$ be a prime that satisfies the congruences $p\equiv \pm 2 \mod 7,\, p \equiv 3 \mod 4$. Then $p$ is unramified in $L$, which equals $K$ in our case. Furthermore $\text{per}(\Frob_{\mathfrak{p}})=(123456)^{-1}$ for any prime $\mathfrak{p}$ of $L$ lying over $p$. Thus with Theorem \ref{ThmCMRed} and Theorem \ref{ThmTransa1} we conclude that $\mathfrak{H}_3$ and $\mathcal{S}_3$ intersect non-transversally in the moduli point $[\overline{C}]$.
\end{proof}
It is worthwile to remark that the curve $C$ from Example \ref{ExmplTouch} also has good supersingular reduction at the primes satisfying
$$p\equiv \pm 3 \mod 7,\, p \equiv 3 \mod 4\,.$$
However, in this case the assumptions of Theorem \ref{ThmCMRed} are not satisfied because one has $\text{per}(\Frob_{\mathfrak{p}})=(123456)$ instead of $(123456)^{-1}$. The proof of the theorem still gives an explicit description of the Cartier operator and shows that $a=1$. However, the criterion iii) of Theorem \ref{ThmTransa1} is never satisfied, i.e., there is no touchpoint for the Cartier operator. As a consequence the moduli point $[\overline{C}]$ is a point of transversal interection when the congruences
$p\equiv \pm 3 \mod 7,\, p \equiv 3 \mod 4$ hold.
\subsection{One example with $a=3$}
The curve $C: y^2=x^8-1$ over $\mathbb{F}_7$ is supersingular and satisfies $a(\Jac(C))=3$ (this can be checked by computing the Hasse-Witt matrix, for example). As a consequence of Lemma \ref{LemIrredComp} the formal completion of $\mathcal{S}_3$ at $[C]$ has $p^5+p^2+1=16857$ irreducible components. One can show that $p^2=49$ of these have the property that they intersect $\mathfrak{H}_3$ non-transversally at the moduli point $[C]$.
\par To prove this, it does not suffice to compute the Cartier-Manin matrix as one would do in the $a=1$ case. In general one must know the matrix of Frobenius on crystalline cohomology mod $p^2$, i.e., $H^1_{\text{crys}}(C/W_2)$. But, if $C$ is defined over $\mathbb{F}_{p^2}$, as in the example, it suffices to look at the matrix of Frobenius acting on $H^1_{\text{dR}}$.
\par These matrices can be computed with Kedlaya's algorithm \cite{Kedlaya}. The computations related to this example were carried out with the aid of \texttt{SAGE} \cite{Sage}. The source code is available on the author's webpage \cite{Source}.
\sloppy
\bibliographystyle{plain}
\bibliography{bibl.bib}
\end{document}